\RequirePackage{etoolbox}
\csdef{input@path}{{style/}{graphics/}}
\documentclass[aap,MSNbibl,dvips]{arximspdf}
\makeatletter
   \@ifpackageloaded{graphicx}{}{\usepackage{graphicx}}
\makeatother
\usepackage{stfloats}

\doi{10.1214/14-AAP1033} 
\volume{25}
\issue{3}
\pubyear{2015}
\firstpage{1650}
\lastpage{1685}

\makeatletter
\fnbelowfloat
\newcommand{\eqref}[1]{(\ref{#1})}
\newcommand{\zset}{{\mathbb Z}}

\newcommand{\midpoints}{\Lambda}
\def\L{\Lambda}
\def\O{\Omega}
\def\Lmn{\Lambda_{m,n}}
\def\Omn{\Omega_{m,n}}

\def\rset{\mathbb{R}}
\def\s{\sigma}

\newtheorem{theorem}{Theorem}[section]
\newtheorem{lemma}[theorem]{Lemma}
\newtheorem{proposition}[theorem]{Proposition}

\newtheorem{corollary}[theorem]{Corollary}
\newproclaim{remark}[theorem]{Remark}
\newtheorem{conjecture}[theorem]{Conjecture}

\newcommand{\cC}{\mathcal C}
\newcommand{\cD}{\mathcal D}
\newcommand{\cE}{\mathcal E}
\newcommand{\cG}{\mathcal G}
\newcommand{\cS}{\mathcal S}

\newcommand{\bbE}{{\mathbb E}}
\newcommand{\bbN}{{\mathbb N}}
\newcommand{\bbP}{{\mathbb P}}
\newcommand{\bbR}{{\mathbb R}}

\newcommand{\si}{\sigma}
\newcommand{\tmix}{T_{\mathrm{mix}}}

\def\a{\alpha}
\def\e{\varepsilon}
\def\k{\kappa}
\def\s{\sigma}
\def\t{\tau}
\def\D{\Delta}
\def\G{\Gamma}
\def\L{\Lambda}
\def\O{\Omega}
\def\Y{\Upsilon}
\makeatother

\begin{document}
\begin{frontmatter}

\title{Random lattice triangulations: Structure and algorithms}
\runtitle{Random lattice triangulations}

\begin{aug}
\author[A]{\fnms{Pietro}~\snm{Caputo}\corref{}\thanksref{T1}\ead[label=e1]{caputo@mat.uniroma3.it}},
\author[A]{\fnms{Fabio}~\snm{Martinelli}\thanksref{T1}\ead[label=e2]{martin@mat.uniroma3.it}},
\author[B]{\fnms{Alistair}~\snm{Sinclair}\thanksref{T3}\ead[label=e3]{sinclair@cs.berkeley.edu}}
\and
\author[A]{\fnms{Alexandre}~\snm{Stauffer}\thanksref{T1}\ead[label=e4]{alexandrestauffer@gmail.com}}

\thankstext{T1}{Supported by the European Research Council through AdG
``PTRELSS'' 228032.}
\thankstext{T3}{Supported in part by NSF Grant CCF-1016896 and by the
European Research Council
through AdG ``PTRELSS'' 228032. Part of this work was done while this
author was visiting University of Roma Tre.}
\runauthor{Caputo, Martinelli, Sinclair and Stauffer}

\affiliation{University of Roma Tre, University of Roma Tre, University
of California, Berkeley and University of Roma Tre}

\address[A]{P. Caputo\\
F. Martinelli\\
A. Stauffer\\
Department of Mathematics\\
University of Roma Tre\\
Largo San Murialdo~1, 00146~Roma\\
Italy\\
\printead{e1}\\
\phantom{E-mail:\ }\printead*{e2}\\
\phantom{E-mail:\ }\printead*{e4}}
\address[B]{A. Sinclair\\
Computer Science Division\\
University of California, Berkeley\\
Berkeley, California 94720-1776\\
USA\\
\printead{e3}}
\end{aug}

\received{\smonth{11} \syear{2012}}
\revised{\smonth{4} \syear{2014}}

%
\begin{abstract}
The paper concerns \textit{lattice triangulations}, that is,
triangulations of the integer points
in a polygon in~$\rset^2$ whose vertices are also integer points.
Lattice triangulations
have been studied extensively both as geometric objects in their own
right and by virtue
of applications in algebraic geometry. Our focus is on random
triangulations in which
a triangulation~$\sigma$ has weight $\lambda^{|\sigma|}$, where
$\lambda
$ is a
positive real parameter, and $|\sigma|$ is the total length of the
edges in~$\sigma$.
Empirically, this model exhibits a ``phase transition'' at $\lambda=1$
(corresponding
to the uniform distribution): for $\lambda<1$ distant edges behave
essentially independently,
while for $\lambda>1$ very large regions of aligned edges appear. We
substantiate
this picture as follows. For $\lambda<1$ sufficiently small, we show
that correlations
between edges decay exponentially with distance (suitably defined), and
also that the
\textit{Glauber dynamics} (a local Markov chain based on flipping
edges) is rapidly mixing
(in time polynomial in the number of edges in the triangulation). This
dynamics has been
proposed by several authors as an algorithm for generating random
triangulations.
By contrast, for $\lambda>1$ we show that the mixing time is
exponential. These are apparently
the first rigorous quantitative results on the structure and dynamics
of random lattice
triangulations.
\end{abstract}\vspace*{-8pt}

%

\end{frontmatter}
%

\section{Introduction}\label{sec:intro}
\subsection{Background}
Let $\Lambda^0_{m,n}= \{0,1,\ldots,m\}\times\{0,1,\ldots,n\}$
denote the set of
integer points
in an $m\times n$ rectangle in~$\rset^2$. This paper is concerned with\vadjust{\goodbreak}
\textit{(full)
triangulations} of $\Lambda^0_{m,n}$, that is, triangulations that use
all the
points. Any such
triangulation partitions the rectangle into $2mn$ triangles, each
triangle being
\textit{unimodular} (having area~$1/2$). See Figure~\ref{fig:example} for
an example
of a triangulation with $m=5$, $n=7$.
Most of our results extend to the case of triangulations of the integer
points in an
arbitrary (not necessarily convex) lattice polygon.

\begin{figure}

\includegraphics{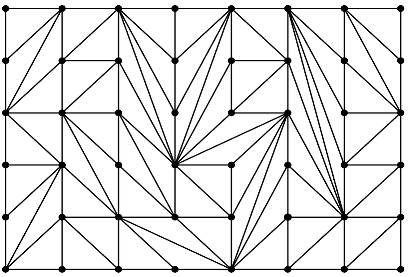}

\caption{A triangulation of a $5\times7$ region.}
\label{fig:example}
\end{figure}

%
%
%

Lattice triangulations are fundamental discrete geometric objects that
have a rich
and beautiful structure; see, for example, \cite{DeLoera} for
background. For example,
a~number of recent papers have developed ingenious combinatorial arguments
to estimate the asymptotic number of triangulations \cite
{Anclin,KZ,MVW,Orevkov},
as well as to explore their connectivity properties under natural local
moves \cite{Eppstein,KZ}.

Lattice triangulations have also received much attention in algebraic
geometry, through
connections with plane algebraic curves and Hilbert's 16th
problem~\cite
{Viro}, the theory of
discriminants~\cite{GKZ} and toric varieties~\cite{Dais}. In several of
these contexts one is
chiefly interested in properties of ``typical'' triangulations, such as
whether they are
\textit{regular} (in the sense of being representable by a ``nice'' lifting
function~\cite{Sturmfels,Ziegler}), whether they contain ``long, thin''
triangles, etc.
This leads one to investigate random triangulations drawn from a uniform
distribution~\cite{KZ,Welzl}.

A natural generalization that can yield useful insights here is to
introduce weights:
specifically, we consider the distribution in which a
triangulation~$\sigma$ has
weight~$\lambda^{|\sigma|}$, where $|\sigma|$ is the total $\ell
_1$~length of the
edges in~$\sigma$, and $\lambda>0$ is a real parameter. The case
$\lambda=1$
is the uniform distribution, while $\lambda<1$ (resp., $\lambda>1$) favors
triangulations with shorter (resp., longer) edges. Figure~\ref
{fig:simulations} shows
typical triangulations of a $50\times50$ region for three values
of~$\lambda$. This
weighted version corresponds rather closely to a model that has
recently been
studied in statistical physics~\cite{RM}.

\begin{figure}

\includegraphics{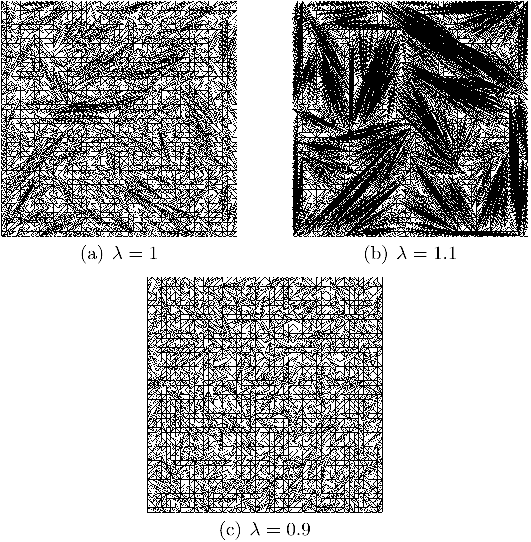}

\caption{Random triangulations of a $50\times50$ lattice.}
\label{fig:simulations}
\end{figure}

One striking feature of the pictures in Figure~\ref{fig:simulations} is
the tendency in
the case $\lambda>1$ for edges to line up in macroscopically large
regions of similar
slope, while for $\lambda<1$ these regions disappear. The uniform case
$\lambda=1$
appears to represent a ``phase transition'' between these two regimes,
in which fairly
large regions form but are unstable. One goal of this paper is to
initiate a rigorous,
quantitative study of this phenomenon.

How does one generate a random triangulation similar to those shown in
Figure~\ref{fig:simulations}? The only known feasible approach for
large values
of $m,n$ is to simulate a local Markov chain, or \textit{Glauber
dynamics}, which converges
to the desired probability distribution over triangulations.
Fortunately there is a very natural
dynamics here: pick a random edge of the triangulation, and if it is
the diagonal of a
convex quadrilateral (which is in fact always a parallelogram), flip it
to the opposite
diagonal. The induced graph on the set of triangulations, in which two
triangulations
are adjacent if and only if they differ by one edge flip, is known as
the ``flip graph.'' The flip graph
is well known to be connected~\cite{Lawson} (indeed, quite a bit more
is known
about its geometry~\cite{Eppstein,KZ}),
so the Glauber dynamics converges to the uniform distribution over
triangulations.
Moreover, a standard ``heat-bath'' modification of the flip
probabilities achieves the
weighted distribution $\lambda^{|\sigma|}$ for any desired~$\lambda$.

This leads to a fundamental question of both algorithmic and structural
interest, posed
by Welzl and others~\cite{Welzl}: what is the \textit{mixing time} of
the Glauber dynamics,
that is, the number of random flips until the dynamics is close to the
stationary distribution?
The pictures in Figure~\ref{fig:simulations}, together with the actual
Glauber dynamics
simulations used to generate them, suggest the following conjecture.
When $\lambda<1$
there are no long-range correlations between edges, and we would expect
the mixing
time to be small; when $\lambda>1$, the large ``frozen'' regions of
aligned edges take
a very long time to break up, and consequently the mixing time should
be large.
We may summarize this in:

%
\begin{conjecture}\label{conj:mt}
\textup{(a)} For any fixed $\lambda<1$, the mixing time of the Glauber dynamics
is $\operatorname{poly}(m+n)$
[specifically, $O(mn(m+n))$].
\textup{(b)} For any fixed $\lambda>1$, the mixing time is exponential in $(m+n)$
[specifically, $\exp(\Omega(mn(m+n)))$].
\end{conjecture}

The ``critical'' case $\lambda=1$ appears to be rather delicate: while
large regions tend
to form, they tend not to be stable over long time scales, and the
behavior of the mixing
time is less clear.

A second goal of this paper is to make the first rigorous progress in
analyzing the mixing
time of these dynamics, and in particular toward a proof of
Conjecture~\ref{conj:mt}.

Before describing our results, we relate random triangulations to
classical spin systems
in statistical physics and explain why they present a challenge to
existing analysis
techniques. It is well known that, in any triangulation of~$\Lambda
^0_{m,n}$, the
\textit{midpoints}
of all the edges are fixed, and are precisely the points of the
half-integer lattice
(with the original lattice points removed)~\cite{Anclin}; see the
example in Figure~\ref{fig:example}.
Thus we can view any triangulation as an assignment of ``spin values''
(edges) to each
of these midpoints, subject of course to consistency constraints.
One might hope, then, to capitalize on the vast literature on both the
structure and
dynamics of spin systems. However, there is a crucial difference here.
In a classical
spin system on a graph, all interactions are local; that is, the
distribution of the value of a given
spin is determined by the spins of its neighbors in the graph. For
triangulations, on the
other hand, each edge belongs to two triangles and thus has four
neighboring edges---but
the midpoints of these edges depend on the triangulation and may lie
very far from the
midpoint of the edge itself. In other words, the geometry of a
triangulation seems to have little
to do with the Cartesian geometry of the lattice. It is this
nonlocality that makes the
application of established techniques very challenging here.
\subsection{Contributions}
Our first results confirm the empirical observation that, in the regime
$\lambda<1$,
correlations between edges decay rapidly with distance---a~property
often referred
to in the spin systems literature as ``strong spatial mixing.''
Specifically, for any sufficiently
small $\lambda<1$, we show that if the configuration of the edge at
midpoint~$x$
is fixed to~$\sigma_x$, then the effect on the distribution of the
edges at other midpoints
$z\in A$, for an arbitrary subset~$A$,
decays exponentially with a natural distance $d(A,\sigma_x)$
between~$A$ and~$\sigma_x$.
[See Section~\ref{sec:equilibrium} for the definition of $d(\cdot,\cdot)$.]
Note that $d(A,\sigma_x)$ depends on the edge~$\sigma_x$ and is not
just the geometric
distance between~$A$ and the point~$x$. Indeed, in contrast to the
notion of strong spatial
mixing for spin systems, or equivalently, Dobrushin's complete
analyticity (see, e.g., \cite{MartOli}),
in this setting, because of the geometric constraints, we cannot expect
a bound that is independent
of the conditioning edge~$\sigma_x$.

%
\begin{theorem}\label{thm:zero}
There exists $\lambda_0>0$ such that for all $\lambda<\lambda_0$ the
following holds.
Let $\mu$ denote the $\lambda$-weighted distribution on triangulations
of $\Lambda^0_{m,n}$,
and $\mu^{\sigma_x}$ the distribution \textit{conditional on the edge
at~$x$ having
configuration~$\sigma_x$}. Then the variation distance between $\mu$
and $\mu^{\sigma_x}$
at any set of midpoints~$A$ satisfies $\Vert\mu-\mu^{\sigma_x}\Vert_A
\leq|A|\exp(-c d(A,\sigma_x))$,
where the constant $c>0$ depends only on~$\lambda$.
\end{theorem}

We remark that this result holds in the presence of arbitrary fixed
edges, or \textit{constraints},
and therefore for triangulations of arbitrary lattice polygons.
Handling constraints introduces
some technical complications into our proofs, and necessitates in
particular the understanding
of minimum length, or ``ground state'' triangulations. (In the absence
of constraints, ground
state triangulations are trivial: all edges are either horizontal,
vertical or unit diagonal.)
Fortunately, as we show, ground state triangulations with constraints
can be constructed
greedily, by placing each edge independently in its minimum length
configuration consistent
with the constraints. We also prove another fundamental property in
this regime, namely
that the probability of an edge exceeding its ground state length
by~$k$ decays exponentially
with~$k$.

Theorem~\ref{thm:zero} is analogous to spatial mixing results in
classical spin systems,
with the additional twist of the distance function~$d$ whose definition
is tailored to the geometry
of triangulations. Our argument is reminiscent of classical
``Peierls-type'' arguments for the
Ising model, which explains why we require $\lambda$ to be sufficiently
small. We conjecture
that the theorem holds for all $\lambda<1$, but handling values of
$\lambda$ close to~1 will
require more sophisticated methods.

We then turn to the Glauber dynamics, again for the regime $\lambda<1$.
Here we can show
the following.

%
\begin{theorem}\label{thm:one}
There exists $\lambda_1>0$ such that for all $\lambda<\lambda_1$ the
mixing time
of the Glauber dynamics on triangulations of $\Lambda^0_{m,n}$ with
parameter~$\lambda$
(and in the presence of arbitrary constraints) is $O(mn(m+n))$.
\end{theorem}

Theorem~\ref{thm:one} goes some way toward proving part~(a) of
Conjecture~\ref{conj:mt},
though it does require $\lambda$ to be sufficiently small (rather than
merely less than~1).
Our proof here uses a path coupling argument with a suitably chosen
exponential metric
on triangulations, and exploits properties of the geometry of the flip
graph which
we also establish. We also prove that our bound on the mixing time is tight.
In the special case of triangulations of a 1-dimensional region
$\Lambda^0_{1,n}
$, we show
that the mixing time is $O(n^2)$ \textit{for all} $\lambda<1$. In
this case
triangulations turn out to be isomorphic to ``lattice paths'' of
length~$2n$ starting and ending
at~0 with $\pm1$ increments at each step.

Our final main result concerns the Glauber dynamics in the complementary
regime $\lambda>1$. Here we are
able to prove a slightly weaker version of part~(b) of Conjecture~\ref
{conj:mt}.

\begin{theorem}\label{thm:two}
For any $\lambda>1$, the mixing time of the Glauber dynamics
on triangulations of $\Lambda^0_{m,n}$ with parameter~$\lambda$ is
$\exp
(\Omega(m+n))$.
\end{theorem}

While this result confirms the conjectured exponential slowdown for all
\mbox{$\lambda>1$}, we believe that the mixing time should be exponential in the
maximum total edge length~$mn(m+n)$ rather than in the maximum length of
a single edge~$(m+n)$. Theorem~\ref{thm:two}
is proved by exhibiting an explicit bottleneck in the dynamics,
specifically an
initial configuration from which it takes a very long time to change the
slope of some long edge from positive to negative.
We also prove a stronger lower bound $\exp(\Omega(n^2/m))$,
which holds whenever $m\ll\sqrt{n}$.

\subsection{Related work}\label{subsec:related}
Triangulations of general point sets are a large topic with numerous
applications
in mathematics and computer science, including combinatorics, optimization,
algebraic geometry, computational geometry and scientific computing.
For excellent
background the reader is referred to the survey by Lee~\cite{Lee} and
the recent book
of DeLoera, Rambau and Santos~\cite{DeLoera}, which also discusses lattice
triangulations in some depth. As mentioned earlier, lattice
triangulations have
been studied both in their own right as geometric objects, and in
several contexts
in algebraic geometry \cite{Dais,GKZ,Sturmfels,Viro,Ziegler}.

Much of the work on lattice triangulations has focused on counting them.
A sequence of beautiful combinatorial arguments \cite{Orevkov,Anclin,KZ,MVW}
has shown that the number of triangulations of $\Lambda^0_{m,n}$ is at most
$O(6.86^{mn})$~\cite{MVW} and at least $\Omega(4.15^{mn})$~\cite{KZ}.
Although it is not the tightest upper bound known, we briefly mention
the elegant
result of Anclin~\cite{Anclin}, who shows that if the edges of a triangulation
of~$\Lambda^0_{m,n}$ are added one by one, starting at the top left
and proceeding
left-to-right and top-to-bottom (by midpoint), then the maximum number of
choices for each edge is two. Since there are fewer than $3mn$ interior
edges, this
immediately yields an upper bound of $8^{mn}$ on the number of triangulations.
(In contrast, the best known upper bound for the number of
triangulations of
a general set of $n$ points in the plane is $43^n$~\cite{ShW}.)

The flip graph on triangulations of $\Lambda^0_{m,n}$ has also been
studied in
some depth,
though we should stress that our work is apparently the first to handle
constraints
(and thus general lattice polygons). For example, the flip graph is
known to have
diameter $O(mn(m+n))$ and to be isometrically
embeddable into the hypercube~\cite{Eppstein}. Random walk on this
graph has
been proposed by several authors as an algorithm for generating random
triangulations,
and has been used heuristically in the formulation of conjectures regarding
typical triangulations~\cite{KZ,Welzl}. However, nothing is known rigorously
about its mixing time. (We note in passing that the analogous flip
dynamics for
triangulations of (the vertex set of) a convex polygon has been extensively
analyzed~\cite{McST,MRS}; however, that problem is much easier because
the set of triangulations has a Catalan structure.)

Random lattice triangulations with weights have appeared in slightly
different form
as a model in statistical physics~\cite{RM}; there, triangulations are weighted
according to the sum of squares of their vertex degrees. This is
actually very
close to our model, as high degrees are associated with long edges. The
structural results reported in~\cite{RM} are based on simulations and
are nonrigorous.
Rigorous results have been obtained for a loosely related
``topological glass'' model by Eckmann and Younan~\cite{Eckmann}.

The literature on structural properties and Glauber dynamics of lattice
spin systems is too vast to summarize here. We refer the reader
to the standard references~\cite{Simon} for structural properties such as
spatial mixing, and~\cite{Fabio} for mixing times of Glauber dynamics.
As explained earlier, while triangulations may be viewed as a spin system,
their geometry is very different from that of a traditional spin system
on the lattice;
our paper can be seen as a first step toward obtaining structural and mixing
time results for triangulations analogous to those for classical spin systems.

We mention finally that our mixing time result for the special case of
\mbox{1-dimensional}
regions with $\lambda<1$ is related to work of Greenberg, Pascoe and
Randall~\cite{GPR} on lattice paths. Those authors use a similar path coupling
argument, but for a different probability distribution on lattice
paths: in their
model paths are biased according to the area under the path, while in ours
the bias depends on the excursions of the path from a fixed line.
\section{The model}
\subsection{Lattice triangulations}\label{subsec:triangs}
Let $\Lambda^0_{m,n}$ denote the set of points in the $m\times n$
region of the
integer lattice $\zset^2$, that is, $\Lambda^0_{m,n}:=\{0,1,\ldots,m\}\times\{
0,1,\ldots,n\}$.
A~\textit{(full) triangulation} of~$\Lambda^0_{m,n}$ is any maximal set of
noncrossing edges
(straight line segments), each of which connects two points of~$\Lambda
^0_{m,n}$ and
passes through no other point. We write $\Omn$ for the set of triangulations
of~$\Lambda^0_{m,n}$. (We will drop
the subscripts on~$\Lambda^0$ and $\O$ when their values are not important.
Also, we will occasionally abuse notation by using $\Lambda^0_{m,n}$
to refer to
the geometric region of~$\rset^2$ that is the convex hull of the integer
points.\setcounter{footnote}{2}\footnote{For notational convenience we shall adopt the slightly
nonstandard convention that $(i,j)$ denotes the point in $\rset^2$ with
vertical coordinate~$i$ and horizontal coordinate~$j$. Thus the region
$\Lambda^0_{m,n}$ has vertical and horizontal dimensions $m$ and~$n$,
respectively.})
In this section we spell out a few basic properties of~$\Omn$ that are
either well known (see, e.g., \cite{DeLoera}) or easily deduced.

As for any point set in~$\rset^2$, the numbers of triangles and edges in
every triangulation of~$\Lambda^0_{m,n}$ are determined by the total
number of points
and the number of points on the boundary of the convex hull of the
point set. Thus every
triangulation in $\Omn$ has $(m+1)(n+1)$ points, $2mn$ triangles,
and $3mn+m+n$ edges, $2(m+n)$ of which are on the boundary and
the remainder in the interior.
Moreover, since $\Lambda^0_{m,n}$ consists of the integer
points inside a lattice polygon, it follows from Pick's theorem that
every triangle is \textit{unimodular}, that is, has area~$1\over2$.

In any triangulation, the \textit{midpoints} of the edges are precisely
the points of
the half-integer lattice with the integer points removed, that is,
\[
\Lmn:=  \bigl\{0,\tfrac{1}{2},1,\tfrac{3}{2},\ldots,m-
\tfrac{1}{2},m \bigr\} \times \bigl\{0,\tfrac{1}{2},1,
\tfrac{3}{2}, \ldots,n-\tfrac{1}{2},n \bigr\} \setminus
\Lambda^0_{m,n}.
\]
Thus we may think of a triangulation as an assignment
$\sigma=\{\sigma_x\}_{x\in\Lmn}$ of an edge to each point of~$\L
=\Lmn$.
We call midpoints~$x$ with one integer and one half-integer
coordinate ``Type~1,'' and those with two half-integer coordinates ``Type~2.''
Figure~\ref{fig:range} shows some of the possible configurations for
the edge
at Types~1 and~2 points~$x$.

\begin{figure}[b]

\includegraphics{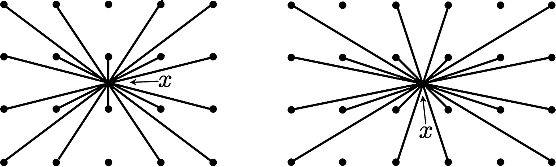}

\caption{Possible configurations of the edge at~$x$. The configurations
differ slightly according to whether~$x$ is Type~1 (left-hand
figure) or Type~2 (right-hand figure). For Type~1 points
the minimum length configuration is vertical (or horizontal); for
Type~2 it is a unit diagonal.}
\label{fig:range}
\end{figure}

Note that the minimum length configuration of the edge
at a Type~1 point is horizontal or vertical, and at a Type~2 point it
is a unit
diagonal.
The configurations of the boundary edges at points
$x=(0,j), (m,j)$ [resp., $x=(i,0),(i,n)$] are forced to be horizontal
(resp., vertical) in all triangulations. We call triangulations in
which all
edges have minimal length \textit{ground state} triangulations. There are
exactly $2^{mn}$ ground state triangulations, which are equivalent up to
flipping of unit diagonals.

We will also consider a generalization to triangulations with the
configurations of some interior edges fixed. Let $\Delta$ be an
arbitrary subset of~$\L$, and $\eta=\{\eta_x\}_{x\in\Delta}$ an arbitrary
assignment of edges to the points of~$\Delta$ that is consistent; that
is, none of
the edges cross. We denote by $\O(\eta,\Delta)$ the set of triangulations
$\sigma\in\O$ that agree with~$\eta$ on~$\Delta$, that is, all possible
completions
of the partial triangulation~$\eta$. Since triangulations are just
maximal consistent
sets of edges, it is clear that for any~$\eta$ at least one such
completion always
exists. We refer to the edges in~$\eta$ as \textit{constraints}.
See Figure~\ref{fig:trfig1}.

\begin{figure}

\includegraphics{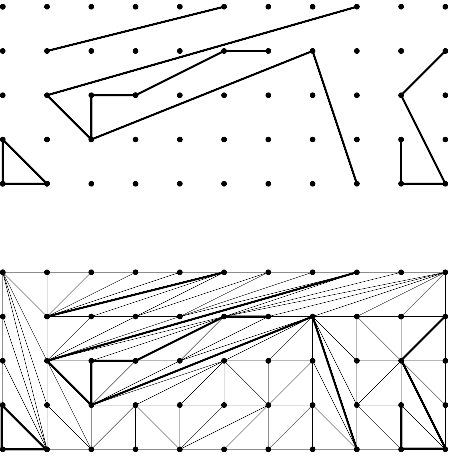}

\caption{Example of a triangulation of a $4\times10$ region. Above:
constraint edges.
Below: a triangulation consistent with the constraints (in bold).}
\label{fig:trfig1}
\end{figure}

An important application of constraints is to triangulations of
arbitrary lattice
polygons (i.e., polygons---not necessarily convex---whose vertices are lattice
points in~$\zset^2$). Let $P$ be the set of integer points in a
lattice polygon.
Then triangulations of the point set~$P$ correspond to triangulations
of an
enclosing rectangle~$\Lambda^0_{m,n}$ with the polygon edges as
constraints (and
an arbitrary fixed triangulation outside the polygon). See Figure~\ref
{fig:polyfig}
for an example.

\begin{figure}[t]

\includegraphics{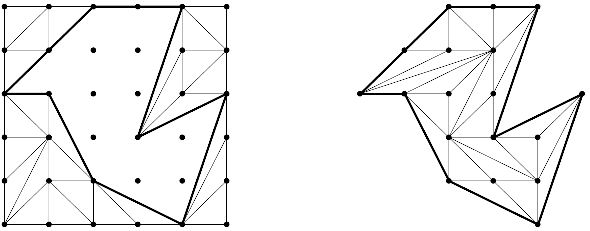}

\caption{Example of a triangulation of a lattice polygon embedded in a
$5\times5$ square.
Left: the polygon edges (in bold), along with the edges of an arbitrary
triangulation
outside the polygon, are fixed as constraints. Right: a triangulation
consistent with the constraints.}
\label{fig:polyfig}
\end{figure}

\begin{figure}[b]

\includegraphics{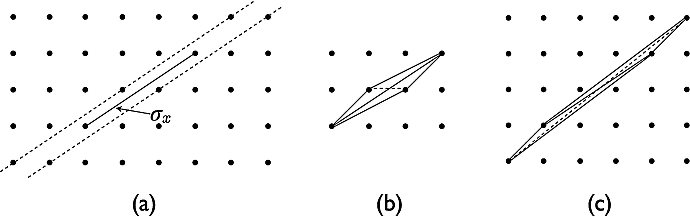}

\caption{Flips of an edge~$\sigma_x$. \textup{(a)}~Possible locations of the
third vertex
of the triangles containing~$\sigma_x$ are the integer points on the
dotted lines.
\textup{(b)}~Unique flip of $\sigma_x$ to a shorter edge ($\sigma_x$ is the
diagonal of
its minimal parallelogram). \textup{(c)}~A flip of~$\sigma_x$ to a longer edge.}
\label{fig:fund}
\end{figure}

%

\subsection{The flip graph}\label{subsec:flipgraph}
Any interior edge~$\sigma_x$ lies in two triangles. From area considerations,
the third vertex of each of these triangles is an integer point on the
line parallel
to~$\sigma_x$ that passes through the closest integer point on either side
of~$\sigma_x$; see Figure~\ref{fig:fund}(a). The integer points on
these lines
occur periodically at intervals equal to the length of~$\sigma_x$,
and are positioned symmetrically on either side of~$\sigma_x$, so that $x$
is the midpoint of the line segment joining a point to its symmetric
pair. The two triangles
containing~$\sigma_x$ form a quadrilateral; this quadrilateral is
convex if and only if the
third vertices of the triangles are a symmetric pair. In this case
$\sigma_x$ is
the diagonal of a parallelogram and thus can be replaced by the other
diagonal to
yield another valid triangulation: $\sigma_x$ is then said to be
\textit{flippable}.
We observe that, for any $\sigma_x$ that is not of minimum length,
there is a~\textit{unique} parallelogram in which $\sigma_x$ is the longer
diagonal; we
call this the \textit{minimal parallelogram} of~$\sigma_x$. Thus
there is a unique flip that takes $\sigma_x$ to a \textit{shorter} edge.
(There are many possible flips that make~$\sigma_x$ longer.) See
Figure~\ref{fig:fund}(b), (c).

We observe also that flipping an edge cannot change its slope from positive
to negative (or vice versa), unless the edge is a unit diagonal. Thus
any sequence
of flips that changes the sign of the slope of an edge must pass
through the minimum
configuration of the edge (horizontal, vertical or unit diagonal).
Figure~\ref{fig:switch} shows how this is achieved.

\begin{figure}[t]

\includegraphics{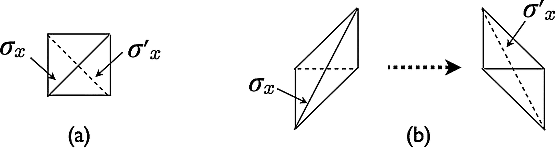}

\caption{Changing the sign of the slope of~$\sigma_x$. \textup{(a)}~If $x$ is Type~2,
the change can be effected only when $\sigma_x$ is a unit diagonal.
\textup{(b)}~If $x$ is Type~1, $\sigma_x$ must first become horizontal (or
vertical), and then be flipped
after the orientation of its surrounding parallelogram has been changed.
In both cases the new configuration is denoted~$\sigma'_x$.}
\label{fig:switch}
\end{figure}

The flipping operation induces a graph on vertex set~$\O$ known as the
``flip graph''
in which two triangulations are adjacent if and only if they differ by
one edge flip. The flip graph
is well known to be connected. To see this, note that a longest edge in
any nonground-state
triangulation is always flippable to a shorter edge because if $\sigma
_x$ is the longest
edge in both of its triangles, then the triangles must form the minimal
parallelogram
of~$\sigma_x$. Hence, for any initial triangulation~$\sigma$, there
exists a sequence
of flips that reaches a ground state. Since flips are reversible, any
triangulation is also
reachable from a ground state. And finally, any ground state is
reachable from any other
by flipping unit diagonals.

For triangulations with constraints we get a flip graph on the smaller
set $\O(\eta,\Delta)$. As we shall see in Section~\ref{subsec:gsl}, in
the presence of
constraints there is still a well-defined ground state in which all
configurations
$\sigma_x$ are of minimal length (subject to the constraints).
Moreover, it can
be shown (see Lemma~\ref{lem:connected} below) that any longest
edge that is not in its ground state configuration is flippable to a
shorter edge,
and hence the flip graph remains connected in this case. In
Proposition~\ref{prop:flipdist}
we shall see that the flip graph actually enjoys a very strong
structural property,
which allows one to exactly compute the shortest path distance (or \textit{flip distance})
between any pair of triangulations.

\subsection{Weighted triangulations and Glauber dynamics}\label
{subsec:glauber}
We associate with a triangulation~$\sigma$ a \textit{weight}~$\lambda
^{|\sigma|}$,
where $\lambda$ is a positive real parameter, and $|\sigma|$ denotes the
\textit{total length} of~$\sigma$, that is, $|\sigma|=\sum_x|\sigma
_x|$; here
$|\cdot|$ denotes the $\ell_1$ (Manhattan) metric. These weights
induce a
probability distribution~$\mu$ on~$\O$ [or, in the presence of constraints,
on $\O(\eta,\Delta)$] via
%
\begin{equation}
\label{eq:gibbs} \mu(\sigma) = \frac{\lambda^{|\sigma|}}{Z},
\end{equation}
where the normalizing factor $Z=\sum_\sigma\lambda^{|\sigma|}$ is the
\textit{partition function}. We refer to~(\ref{eq:gibbs}) as the \textit{Gibbs distribution}.
The case $\lambda=1$ corresponds to the uniform
distribution on triangulations; the cases $\lambda<1$ (resp., $\lambda>1$)
favor shorter (resp., longer) triangulations. Note that when $\lambda<1$
the triangulations of maximum weight are precisely the \textit{ground
state}
triangulations, that is, those that minimize~$|\sigma|$.\footnote{This
explains our
choice of the term ``ground state''; in statistical physics, ground
states are
configurations of minimum energy, and thus maximum weight in the Gibbs
distribution. The term ``ground state'' is less appropriate when
$\lambda>1$,
but our focus in this paper is mainly on the case $\lambda<1$.}

The flip graph forms the basis of a local Markov chain (or \textit{Glauber
dynamics})
on~$\O$ [or, more generally, on $\O(\eta,\Delta)$]
as follows. In state $\sigma\in\O$, pick a point $x\in\L$
uniformly at
random;\footnote{In the presence of constraints, we pick $x$ u.a.r. from the midpoints of
\textit{nonconstraint}
edges.} if the edge~$\sigma_x$ is flippable to edge $\sigma'_x$
(producing a new
triangulation~$\sigma'$), then flip it with probability
\[
\frac{\mu(\sigma')}{\mu(\sigma')+\mu(\sigma)} = \frac{\lambda^{|\sigma'_x|}}{\lambda^{|\sigma'_x|} + \lambda
^{|\sigma
_x|}},
\]
else do nothing. Since the flip graph is connected, this so-called
``heat-bath'' dynamics
defines an ergodic Markov chain on~$\Omega$ [or on $\O(\eta,\Delta)$]
that is reversible with respect to~$\mu$. Hence the dynamics converges
to the stationary
distribution~$\mu$. We will analyze convergence to stationarity
via the standard notion of \textit{mixing time}, defined by
%
\begin{equation}
\label{eq:tmix} \tmix= \inf \Bigl\{ t\in\bbN\dvtx \max_{\si\in\O}
\bigl\|p^t(\si,\cdot) - \mu\bigr\| \leq1/4 \Bigr\},
\end{equation}
where $p^t(\si,\cdot)$ denotes the distribution after $t$ steps when
the initial state is $\si$, and
$\|\nu-\mu\|=\frac{1}2\sum_{\si\in\O}|\nu(\si)-\mu(\si)|$ is
the usual
total variation distance between
two distributions~$\mu,\nu$.

\section{Structural properties}

In this section we establish some basic geometric properties of lattice
triangulations and
the flip graph. Throughout we will work with arbitrary constraints, so
that our results
apply in particular to triangulations of any lattice polygon.

\subsection{The minimal parallelogram and excluded region}
\label{subsec:minpar}
We begin with a useful property of the lattice~$\zset^2$. 
In the sequel the distance of a point from a line will always mean the
usual Euclidean distance.

%
\begin{proposition}\label{prop:closest}
Let $a,b$ be positive and coprime, and consider the infinite line
through $(0,0)$
of slope $(a,b)$. Then the points of~$\zset^2$ that are closest
to this line (and are not on the line) are at horizontal distance~$a^{-1}$
and vertical distance~$b^{-1}$ on either side of the line.
\end{proposition}

\begin{pf}
Consider two successive integer points on the line [say, $(0,0)$ and
$(a,b)$], and the $a$ horizontal lines at integer heights $0,1,2,\ldots,a-1$.
On each such horizontal line, consider the closest integer point to the line
that lies strictly to the right of the line. (By symmetry, the points
on the left of the line behave similarly.)
It is readily verified that the horizontal distances of these points from
the line are all distinct elements of the set $\{\frac
{i}{a}\dvtx i=1,2,\ldots,a\}$.
Hence the closest of these points is at horizontal distance~$a^{-1}$.
This pattern repeats with period~$a$ throughout $\zset^2$.
The same argument applies to the vertical distances, with period~$b$.
Finally, note that a point at horizontal distance~$a^{-1}$ is also at
vertical distance~$b^{-1}$.
\end{pf}

An immediate corollary is the following.

%
\begin{proposition}\label{prop:parallel}
Let $a,b$ be positive and coprime. The infinite family of parallel
lines with slope
$(a,b)$, horizontal separation~$a^{-1}$ (and thus vertical
separation~$b^{-1}$) and such that one line in the family passes
through $(0,0)$ partitions~$\zset^2$, in the sense that every
integer point lies on one of the lines.
\end{proposition}

%

Given any triangulation~$\sigma$, consider an arbitrary edge~$\sigma_x$
that is not horizontal or vertical. Let its slope be $(a,b)$ with $a,b$ coprime,
and denote its upper and lower endpoints by~$p_1,p_2$, respectively.
By Proposition~\ref{prop:closest} and its proof, we may identify the unique
closest integer point to the right of~$\sigma_x$ that lies horizontally
and vertically
between $p_1$ and~$p_2$. Call this point~$p_3$. By symmetry
there is a corresponding closest point~$p_4$ on the other side
of~$\sigma_x$,
so that $x$ is the midpoint of the line segment~$p_3p_4$. The points
$p_1,p_3,p_2,p_4$ form the vertices of the \textit{minimal parallelogram}
of~$\sigma_x$ mentioned earlier: $\sigma_x$~can be flipped to the shorter
edge $p_3p_4$ if and only if all edges of the minimal parallelogram are
present in
the triangulation.

\begin{figure}

\includegraphics{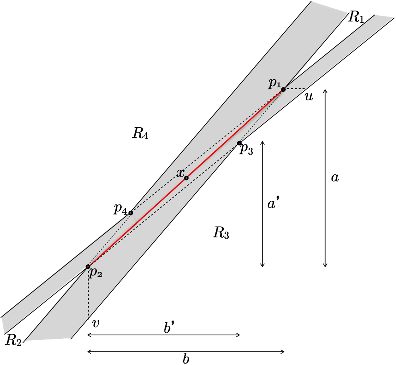}

\caption{The minimal parallelogram and excluded region of an
edge~$\sigma_x$.}
\label{fig:excluded}
\end{figure}

Let $\sigma_x$ be any nonhorizontal and nonvertical
triangulation edge with endpoints $p_1,p_2$ and minimal
parallelogram $p_1p_3p_2p_4$. We define the \textit{excluded region}
of $\sigma_x$ as the union of two strips with parallel sides that
extend to the boundaries of~$\Lambda^0$, whose intersection
is the minimal parallelogram of~$\sigma_x$; see the shaded region
in Figure~\ref{fig:excluded}.
The complement of the excluded region consists of four disjoint closed cones,
which we shall call $R_1,R_2,R_3,R_4$, as shown in Figure~\ref{fig:excluded}
(so the apex of~$R_i$ is~$p_i$).
The following proposition establishes that the excluded region
is free of integer points.

\begin{proposition}\label{prop:excluded}
The excluded region of any nonhorizontal and nonvertical
triangulation edge does
not contain any integer points in its interior.
\end{proposition}

%

\begin{pf}
Consider an arbitrary triangulation edge $\sigma_x$, not horizontal or
vertical,
with endpoints $p_1,p_2$ and minimal
parallelogram $p_1p_3p_2p_4$, as illustrated in Figure~\ref{fig:excluded}.
It suffices to show that each of the two pairs
of parallel lines that bound the strips in Figure~\ref{fig:excluded}
are adjacent lines in a family
as in Proposition~\ref{prop:parallel}.
Consider first the pair of lines through $p_1p_4$ and $p_2p_3$. We need
to verify
that~$p_1$ is a closest point to the line through $p_2p_3$. Assume
w.l.o.g. that
$p_2$ is the point~$(0,0)$, and $p_1$ is the point~$(a,b)$, with $a,b$
positive and
coprime. Suppose $p_3$ is the point $(a',b')$, and recall that the
horizontal distance
of~$p_3$ from~$\sigma_x$ is~$a^{-1}$. By similarity of triangles, the
horizontal
distance of~$p_1$ from the line through $p_2p_3$ (distance $p_1u$ in
Figure~\ref{fig:excluded})
is $\frac{a}{a'}\times\frac{1}{a}=\frac{1}{a'}$. Hence, by
Proposition~\ref{prop:closest},
$p_1$ is a closest point to this line, and by Proposition~\ref
{prop:parallel} this
parallel strip contains no integer points.

An analogous argument using vertical separation shows that the distance
$p_2v$ is~$\frac{1}{b'}$ and hence the other
parallel strip, bounded by lines through $p_2p_4$ and $p_1p_3$, also
contains no
integer points. This completes the proof.
\end{pf}

\subsection{The ground state lemma}
\label{subsec:gsl}
We turn now to the structure of the \textit{ground state}
triangulations, that is, those of minimum
total length. These are the triangulations of maximum weight in the probability
distribution~(\ref{eq:gibbs}) when $\lambda<1$, and they play a central
role in our analysis
of spatial mixing in Section~\ref{sec:equilibrium}.
In the absence of constraints, the ground state triangulations are
very simple: every edge is either horizontal or vertical or a unit
diagonal, so in particular,
the ground state is unique up to flipping of the unit diagonals.
The presence of constraint edges, however, may change the ground
state considerably. In this subsection we prove that the ground state remains
(essentially) unique, and can be described easily; see Figure~\ref
{fig:trfig1gs} for an example.

\begin{lemma}[(Ground state lemma)]\label{lemma:gsl}
Given any set of constraint edges, the ground state triangulation is unique
(up to possible flipping of unit diagonals) and can be constructed by
placing each edge in its minimal length configuration consistent with the
constraints, independent of the other edges.
\end{lemma}

\begin{figure}

\includegraphics{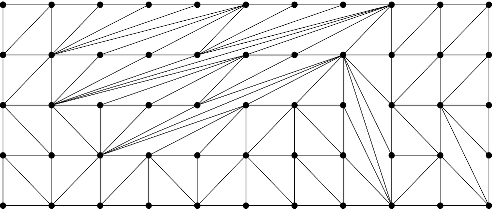}

\caption{Example of a ground state triangulation in the presence of the
constraints given in Figure~\protect\ref{fig:trfig1}.}
\label{fig:trfig1gs}
\end{figure}

Lemma~\ref{lemma:gsl} will follow from a short sequence of observations.
Given a set of constraints~$(\eta,\Delta)$, we call a midpoint $x$
\textit{constrained} if the minimal configuration of the edge at $x$ consistent
with the constraints is not horizontal or
vertical or unit diagonal, or if the minimal configuration is a unit
diagonal and its other unit diagonal configuration is not consistent.

Let $\sigma_x$ be a consistent configuration for the edge at~$x$ with endpoints
$p_1,p_2$ as in Figure~\ref{fig:excluded}. 
We say that $\sigma_x$ is
\textit{spanned} by a constraint edge~$C$ if: (i) $\s_x$ and $C$ are
compatible
(i.e., they do not cross each other); and (ii) $C$ crosses the line segment
$p_3p_4$ in Figure~\ref{fig:excluded} (i.e., the shorter diagonal of the
minimal parallelogram of~$\sigma_x$). Note that this necessarily
implies that the endpoints of~$C$ lie in the regions~$R_1$ and~$R_2$.

\begin{proposition}\label{prop:minimal}
Given a set of constraints, a consistent configuration~$\sigma_x$ for
the edge
at a constrained $x$ is minimal if and only if $\sigma_x$ is spanned by
a constraint
edge. Moreover, the minimal configuration $\sigma_x=\bar\sigma_x$ is unique.
\end{proposition}

\begin{pf}
Assume first that $\sigma_x$ is minimal, and $x$ is constrained. Then
in particular
the edge $p_3p_4$ must not be a consistent configuration for the edge at~$x$
(since either it would be a shorter configuration,
or in the case of a unit diagonal it is explicitly prohibited). Hence a
constraint
edge must pass through~$p_3p_4$.

For the other direction,
suppose $\sigma_x$ is spanned by a constraint edge.
We claim that any consistent configuration for the edge at~$x$ must
have its endpoints
in the cones~$R_1,R_2$ defined by~$\sigma_x$; see Figure~\ref{fig:excluded}.
This immediately implies that $\sigma_x$ is the unique minimal
configuration for
this edge. To see the above claim, let $C$ be the spanning constraint edge.
Since $C$ crosses $p_3p_4$ and does not cross $p_1p_2$, its endpoints
must be in $R_1$ and $R_2$.
Also, since any configuration~$\sigma'_x$ for the edge at~$x$ must be
obtained from~$\sigma_x$ by rotating about its midpoint~$x$, and there
are no integer
points in the excluded region, the endpoints of~$\sigma'_x$ must lie either
in~$R_1,R_2$ or in $R_3,R_4$. But an edge between~$R_3$ and~$R_4$ would
necessarily cross~$C$, so we are left only with the possibility $R_1,R_2$.
\end{pf}

%
%
%
%

\begin{proposition}\label{prop:intersect}
Given a set of constraints, let $\bar\sigma_x,\bar\sigma_y$ be minimal
configurations
of any two distinct edges. Then $\bar\sigma_x,\bar\sigma_y$ do not cross.
\end{proposition}

\begin{pf}
By Proposition~\ref{prop:minimal}, both $\bar\sigma_x$ and $\bar
\sigma
_y$ are
either spanned by a constraint edge or the corresponding midpoints are
unconstrained. If both $x,y$ are unconstrained
then clearly they do not cross. So assume w.l.o.g. that $\bar\sigma
_x$ is
spanned by some constraint edge~$C$. Suppose for contradiction that
$\bar\sigma_y$ crosses~$\bar\sigma_x$. Then $\bar\sigma_y$
must have endpoints either in $R_1,R_2$ or in $R_3,R_4$, where the
$R_i$ are
the regions defined by~$\bar\sigma_x$ as in Figure~\ref{fig:excluded}.
As in the proof of
Proposition~\ref{prop:minimal}, an edge between~$R_3$ and~$R_4$ would
necessarily cross~$C$, so $\bar\sigma_y$ must have endpoints
in~$R_1,R_2$. But clearly this implies that $|\bar\sigma_y|>|\bar
\sigma_x|$.

Now if $y$ is constrained, then switching the roles of~$x$
and $y$ in the above argument yields $|\bar\sigma_x|>|\bar\sigma_y|$,
a contradiction.
And if $y$ is unconstrained, then $\bar\sigma_y$ is horizontal,
vertical or a unit diagonal,
in which case it cannot be strictly longer than the constrained edge
$\bar\sigma_x$,
again a contradiction.
%
\end{pf}

The ground state lemma (Lemma~\ref{lemma:gsl}) now follows immediately from
Propositions~\ref{prop:minimal} and~\ref{prop:intersect}. Given any set
of constraints
$(\eta,\Delta)$, we may thus speak of a minimal length triangulation
$\bar\sigma=\{\bar\sigma_x\}_{x\in\L}$ which is unique except for
possible flipping
of unit diagonals. We call $\bar\sigma$ a \textit{ground state}
triangulation.

\subsection{Flip distance}\label{subsec:flips}
We start with a proof of the claim in Section~\ref{subsec:flipgraph}
that the
flip graph remains connected in the presence of arbitrary constraints.
(Recall that connectedness is easily verified in the absence of constraints.)
We shall also see that the diameter is always small, regardless of the
constraints.

%
\begin{lemma}\label{lem:connected}
For any set of constraints $(\eta,\Delta)$, the flip graph on
triangulations in
$\Omega(\eta,\Delta)$ is connected. Moreover, its diameter is $O(mn(m+n))$.
\end{lemma}

\begin{pf}
By the ground state lemma (Lemma~\ref{lemma:gsl}), there is a
well-defined family of ground state triangulations for any such $(\eta,\Delta)$,
which differ only up to flipping of unit diagonals. As in the
unconstrained case,
it suffices to show that a ground state is reachable from any triangulation
in~$\Omega(\eta,\Delta)$, since by reversibility this implies that any
triangulation
is reachable from a ground state, and ground states are obviously reachable
from each other.
Since in any ground state triangulation all edges have minimal possible length,
it suffices in turn to prove that, in any nonground-state
triangulation, there is
an edge that is flippable to a shorter edge. This we now do.

Let $\sigma\in\Omega(\eta,\Delta)$ be an arbitrary nonground-state
triangulation,
and let $\sigma_x$ be some edge of~$\sigma$ that is not in ground
state. If
$\sigma_x$ is flippable to a shorter edge, we are done; if not, then
$\sigma_x$
cannot be the longest edge in both of its triangles, so we can find a
\textit{longer}
edge, $\sigma_y$, that shares an endpoint with~$\sigma_x$. The crucial
observation
(see the next paragraph) is that: 
(i) $\s_y$ is not a constraint edge; and (ii) $\sigma_y$ is not in
ground state. Thus
we can iterate this process, finding a connected sequence of edges of increasing
length that are not in ground state, until we find one that is longest
in both its triangles.
This edge must be flippable to a shorter edge, and we are done.

To conclude, we verify the above claim about $\sigma_y$.
This follows from the fact that, in any triangle, if the longest edge
is in ground state, including the case where it is a constraint edge,
then so are the other two edges. To see this, first note that by Lemma~\ref{lemma:gsl} one can assume without loss of generality that the
longest edge
in the triangle is a constraint edge. Then the claim follows from
Proposition~\ref{prop:minimal} and the fact that
the longest edge in
a triangle spans the other two edges; this latter fact holds because
the triangle contains
no lattice points, so the longest edge must pass between each of the
other edges and
its closest lattice point on the same side of the edge; see Figure~\ref
{fig:excluded}.

Finally, the assertion concerning the diameter follows immediately from
the fact that
the maximum length of a triangulation is $O(mn(m+n))$, which is clearly
an upper bound
on the number of edge-shortening flips needed to reach a ground state
triangulation.
\end{pf}

Next, we shall compute the \textit{flip distance} between two arbitrary
triangulations
in the flip graph, that is, the minimal number of flips required to
obtain one triangulation
from the other. As usual, fix a set of constraints $(\eta,\D)$.
Given a midpoint~$x$, let $\O_x$ denote the set of all possible values
of the edge
at~$x$ that are consistent with the constraints. This set contains the ground
state~$\bar\si_x$ (or, in the case where~$x$ is unconstrained and of
Type~2, the
pair of ground states corresponding to two opposite unit diagonals).
Say that two elements $\si_x,\si'_x\in\O_x$ are \textit{neighbors},
written $\si_x\sim\si'_x$,
if $\si_x$ is flippable to $\si'_x$ within a valid triangulation $\si
\in
\O(\eta,\D)$.
Recall that for any $\si_x\in\O_x$ there is at most one $\sigma
'_x\sim
\sigma_x$
such that 
$|\si_x|\geq|\si'_x|$. If $\si_x$ is not in ground state, then an edge
$\sigma'_x\in\O_x$
with $\sigma'_x\sim\sigma_x$ and $|\si_x|> |\si'_x|$ necessarily exists.
This follows from the process described in the proof of Lemma~\ref
{lem:connected},
since~$\sigma_x$ must eventually be flippable to a shorter edge.
These observations prove that the resulting graph with vertex set $\O
_x$ is a tree
rooted at the ground state~$\bar\sigma_x$ (or, when the ground state is
not unique,
at a pair of neighboring ground states corresponding to unit diagonals).
Given $\si_x,\tau_x\in\O_x$, let $\k(\si_x,\tau_x)$ be the distance
between $\si_x,\tau_x$
in the tree, that is, the minimal number of flips of the edge at~$x$
required to change its
configuration from~$\sigma_x$ to~$\tau_x$ (regardless of the
disposition of the other edges).
Notice that any path between $\si$ and $\tau$ in the flip graph must
contain all of
the above flips for every~$x$, and thus have length at least $\sum_{x}\k
(\si_x,\tau_x)$.
We call these the \textit{indispensable} flips for the pair $\si,\tau$.

The next result expresses a fundamental structural property of the flip
graph: there
is a path between any two triangulations that consists only of
indispensable flips.
Note that this path always exists, independent of the constraints.
Although we shall not
require this result in our subsequent analysis, we include it here for
completeness.

%
\begin{proposition}\label{prop:flipdist}
The flip distance between any two triangulations $\si,\tau\in\O
(\eta,\D
)$ is equal to
$\sum_{x}\k(\si_x,\tau_x)$.
\end{proposition}

\begin{pf}
Given any pair of triangulations $\si,\tau\in\O(\eta,\D)$ with
$\si\ne
\tau$, we show
how to perform an indispensable flip in either $\si$ or~$\tau$ to
produce a new
pair $\si',\tau'$. Since the set of indispensable flips for $\si
',\tau
'$ differs from that
for $\si,\tau$ only in the single flip just performed, iteration of
this procedure
until $\si=\tau$ produces a path of indispensable flips between~$\si$
and~$\tau$.

To justify the above claim, let $\D'=\{x\in\L\dvtx \si_x=\tau_x\}$;
note that $\D'$ contains~$\D$, and $(\D,\eta)=(\D,\si)$.
We will view $(\D',\si)$ as an expanded set of constraints.
Let $x_*$ denote a point $x\in\L\setminus\D'$ such that
\[
|\si_{x_*}|=\max \bigl\{|\si_x|\dvtx x\in\L\setminus
\D' \bigr\}.
\]
Similarly, let $y_*$ denote a maximal edge in $\tau$ outside~$\D'$.
Assume w.l.o.g. that $|\si_{x_*}| \ge|\tau_{y_*}|$ (else the same argument
follows with the roles of $\si_{x_*}$ and $\tau_{y_*}$ interchanged).
We claim that an indispensable flip can be performed on~$\sigma_{x_*}$.

Note first that $|\si_{x_*}| \ge|\tau_{y_*}| \ge|\tau_{x_*}|$, so
since ground states
are minimal and $\si_{x_*}\ne\tau_{x_*}$, the only way
$\si_{x_*}$ can be in ground state is if $|\si_{x_*}| = |\tau_{x_*}|$
and $\si_{x_*},\tau_{x_*}$
are opposite unit diagonals; in this case an indispensable flip of~$\si
_{x_*}$ must
be possible (since it could be prevented only by a longer
nonconstraint edge, and
$\si_{x_*}$ is assumed to be maximal)
and takes us directly to~$\tau_{x_*}$. Otherwise, $\si_{x_*}$ is not in
ground state,
so from the proof of Lemma~\ref{lem:connected} since $\si_{x_*}$ is a longest
such edge it is flippable to a shorter edge.
To see that this flip is indispensable, note that $\si_{x_*}\ne\tau
_{x_*}$ and
$|\si_{x_*}| \ge|\tau_{x_*}|$, so the path from $\si_{x_*}$ to
$\tau
_{x_*}$ in
the tree~$\O_x$ must include the shortening flip at~$\si_{x_*}$.
This completes the proof.
\end{pf}

%
\subsection{The influence region}\label{subsec:upsilon}
The following notion of ``influence region'' of an edge will play a key
role in our analysis of
decay of correlations under the Gibbs distribution for small~$\lambda$.
As usual we fix an arbitrary set of constraints $(\eta,\D)$.
Let $\bar\sigma$ be a ground state triangulation compatible with the
constraints as in
Lemma~\ref{lemma:gsl}.
Suppose we change the configuration of the edge at some point~$x$ to a
nonground-state
value $\sigma_x\ne\bar\sigma_x$ (still consistent with the
constraints). We would like
to identify the minimal region ``affected by''
the new edge~$\sigma_x$. We call this region the \textit{influence region} of~$\sigma_x$,
and denote it~$\Upsilon(\sigma_x)$.

To define $\Upsilon$ formally, we use the notion of a \textit{ground
state region},
defined as any (not necessarily connected) region of~$\rset^2$ bounded
by ground
state edges.
Then $\Upsilon(\sigma_x)$ is defined as the smallest ground state
region containing~$\sigma_x$; see Figure~\ref{fig:upsilon}.
When $\sigma_x=\bar\sigma_x$, we define $\Upsilon(\sigma_x)$ to be empty.
Observe that $\Upsilon(\sigma_x)$ is a connected ground state region.
Moreover,
since the boundary of~$\Upsilon(\sigma_x)$ must consist of ground
state edges
that do not cross~$\si_x$, the presence of $\sigma_x$ has no effect on
the ground state
outside this region.

\begin{figure}

\includegraphics{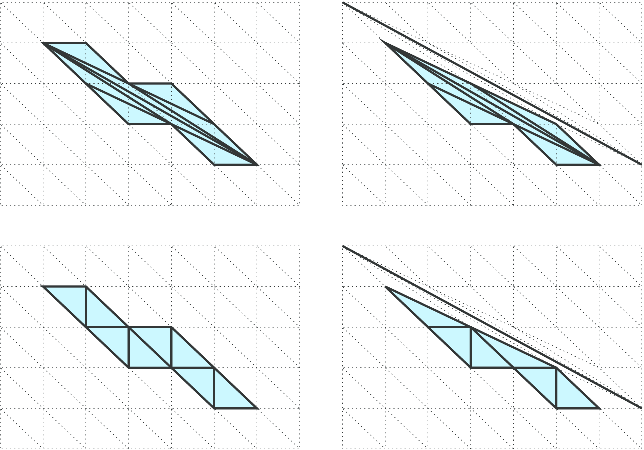}

\caption{The region $\Upsilon(\si_x)$ when $\si_x$ is an edge of slope
$(-3,5)$ (with midpoint in the center of the box).
Left: in the case of no constraints. Right:
in the presence of a constraint edge of slope $(-4,7)$.
Above: the triangles inside correspond to the second definition of the
influence region. Below: the triangles inside
correspond to the first definition of the influence region. In the
background, a ground state triangulation (dotted lines).}
\label{fig:upsilon}
\end{figure}

%

Next, we discuss an alternative construction of the influence region,
and then prove that it coincides with $\Upsilon(\sigma_x)$; see Lemma~\ref{ups*}.
In the setting in the first paragraph above, let $\bar\sigma(\sigma_x)$
denote the new ground state triangulation when the edge~$\sigma_x$ is
added as an
additional constraint. Consider the region $\Upsilon^*(\si
_x)\subseteq
\bbR^2$ defined
by successive addition of triangles as follows. Let $y_1,y_2$ and
$z_1,z_2$ denote the
midpoints of the edges of the minimal parallelogram of~$\sigma_x$, so
that this parallelogram
is made up of the two triangles $T_1=(\si_x,\bar\si_{y_1}(\si
_x),\bar\si
_{y_2}(\si_x))$ and
$T_2=(\si_x,\bar\si_{z_1}(\si_x),\bar\si_{z_2}(\si_x))$. Note that
these triangles are
consistent with the constraints: indeed, if one of the constraint
edges~$C$ were
to cross an edge of one of these triangles, then $C$ would span $\si
_x$; and by
Proposition~\ref{prop:minimal} this would imply that $\si_x$ is
minimal, contradicting
the assumption that $\si_x\neq\bar\si_x$. Thus, we can add the
triangles $T_1$ and $T_2$.
Next, we add a further triangle for each edge $\bar\si_{y_i}(\si_x)$,
$i=1,2$, such
that $\bar\si_{y_i}(\si_x)\neq\bar\si_{y_i}$ and for each edge
$\bar\si
_{z_i}(\si_x)$, $i=1,2$,
such that $\bar\si_{z_i}(\si_x)\neq\bar\si_{z_i}$. For edges that are
already in their
original ground state, we add no further triangle. More precisely,
if $\bar\si_{y_1}(\si_x)\neq\bar\si_{y_1}$, then define midpoints
$y_{11}$, $y_{12}$ and the triangle
$T_{11}=(\bar\si_{y_1}(\si_x),\bar\si_{y_{11}}(\si_x),\bar\si
_{y_{12}}(\si_x))$,
which is one half of the minimal parallelogram of~$\bar\si_{y_1}(\si_x)$.
The same reasoning as above shows that this triangle is consistent with the
constraints as long as $\bar\si_{y_1}(\si_x)\neq\bar\si_{y_1}$.
Similarly, define $y_{21},y_{22}$ and $T_{12}=(\bar\si_{y_2}(\si
_x),\bar
\si_{y_{21}}(\si_x),\bar\si_{y_{22}}(\si_x))$, which is one half
of the
minimal parallelogram of $\bar\si_{y_2}(\si_x)$.
The same construction can be applied to the other side of $\si_x$ to
define the points $z_{11},z_{12},z_{21},z_{22}$ and the associated
triangles $T_{21},T_{22}$.
This branching procedure is iterated until one of the midpoints, say $w$,
is such that $\bar\si_w(\si_x)=\bar\si_w$. 
In this case that branch is stopped, and we continue with the other
available branches, if any.
$\Upsilon^*(\si_x)$~is defined as the union of all triangles added in
this way.

Notice that all triangles added contain at least one edge that is not
in ground state.
However, the boundary of $\Upsilon^*(\si_x)$ is made up only of ground
state edges.
Therefore, $\Upsilon^*(\si_x)$ is a ground state region, and as such it
is also a union of ground state triangles.
(See Figure~\ref{fig:upsilon}.)

\begin{lemma}\label{ups*}
For any point~$x$ and any value of the edge $\si_x$, one has $\Upsilon
^*(\si_x)=\Upsilon(\si_x)$.
\end{lemma}

\begin{pf}
We can assume $\si_x\neq\bar\si_x$, since otherwise $\Upsilon
^*(\si
_x)=\Upsilon(\si_x)=\varnothing$.
Since $\Upsilon^*(\si_x)$ is
a ground state region
that contains $\si_x$, we have
$\Upsilon(\si_x)\subseteq\Upsilon^*(\si_x)$. Next, suppose for
contradiction that $\Upsilon(\si_x)$
is strictly contained in $\Upsilon^*(\si_x)$. Then there must be a
point~$z$ such that
$\bar\si_z$ is a boundary edge of $\Upsilon(\si_x)$ but $z$ is an
interior point
of $\Upsilon^*(\si_x)$. This implies that $z$ can achieve its ground
state value $\bar\si_z$
even in the presence of $\si_x$, but the fact that $z$ is an interior
point of $\Upsilon^*(\si_x)$
implies that $\bar\si_z(\si_x)\neq\bar\si_z$; otherwise the branching
procedure defining
$\Upsilon^*(\si_x)$ would have stopped at $z$, which implies that $z$
is not
an interior point. (Note that the triangles added by the branching
procedure form a tree,
so in this case $\sigma_z$ must indeed lie on the boundary.)
Thus, such a $z$ cannot exist.
\end{pf}

%

\section{Spatial mixing for small \texorpdfstring{$\lambda$}{$lambda$}}\label{sec:equilibrium}
%

In this section we prove some fundamental properties of the Gibbs
distribution~$\mu$
for sufficiently small values $\lambda<1$. Since long edges are
penalized in this regime,
one might expect
that the probability that an edge is longer than its ground state value
should decrease
exponentially in this excess length. We prove this property in
Corollary~\ref{cor1} below,
as a consequence of a rather more general exponential tail property
(Lemma~\ref{expdecS}).
We then go on to establish exponential decay of correlations (or \textit{
spatial mixing}) in the
same regime (see Theorem~\ref{covdecga}), as advertised in
Theorem~\ref
{thm:zero} in
the \hyperref[sec:intro]{Introduction}. It is worth emphasizing that all these properties are
shown to hold
uniformly in the constraints $(\eta,\D)$, and thus in particular for
triangulations of arbitrary
lattice polygons.

We begin with some additional terminology. Fix a set of constraint edges.
For any triangulation~$\sigma$ consistent with the constraints, we call
a triangle
in~$\sigma$ a \textit{G-triangle} if all of its edges are in ground
state (with respect to the
constraints), and a \textit{B-triangle} otherwise.
Two triangles are
\textit{connected} if they share an edge, and \textit{$*$-connected} if
they share a vertex.
Note that every maximal connected region of B-triangles is a
ground state region
(see Figure~\ref{fig:trfig3}), and thus the configuration inside
it can always be replaced by a ground state without affecting the rest
of the triangulation.
We shall use this property crucially in our arguments below.

\begin{figure}

\includegraphics{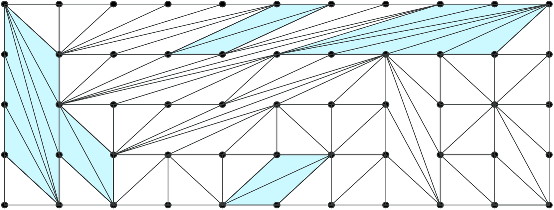}

\caption{The triangulation with constraints of Figure~\protect\ref
{fig:trfig1}, showing its
four connected regions of B-triangles (shaded). All unshaded triangles
are G-triangles. Compare also with Figure~\protect\ref{fig:trfig1gs}.}
\label{fig:trfig3}
\end{figure}

%

We denote by $\cS=\cS(\si)$ the ground state region consisting of
all B-triangles in~$\si$.
The set $\cS$ naturally describes the deviation of~$\sigma$ from
the ground state $\bar\si$.
For any midpoint $x$, we denote by $\cS_x=\cS_x(\si)$ the maximal
connected component of B-triangles containing the edge at~$x$
in $\si$ [so $\cS_x(\sigma)=\varnothing$ if and only if
$\sigma_x$ is contained in two G-triangles].
Notice that $\cS_x$ is a ground state region and that $\cS(\si
)=\bigcup_x S_x(\si)$ for all $\si$.

The next lemma contains the key Peierls-type estimate that we need. As
in \eqref{eq:gibbs}, $\mu$ stands for the Gibbs distribution on
triangulations with parameter~$\lambda$.

%
\begin{lemma}\label{expdecS0}
For all sets of constraints, all points $x\in\L$, all $\lambda>0$
and all ground state regions $S$,
%
\begin{equation}
\label{pereaso} \mu (\cS_x=S )\leq(8\lambda)^{|S|/2}.
\end{equation}
\end{lemma}

\begin{pf}
For any ground state region~$S$,
let $\midpoints(S)$ denote the set of interior midpoints of~$S$, namely
the collection
of $z\in\midpoints\cap S$ that are not on the boundary of~$S$. Notice
that any triangulation~$\t$
of the region~$S$ is obtained by specifying the edges $\t_z$, $z\in
\midpoints(S)$.
We let $\O(S)$ denote the set of all such triangulations, and $\O^*(S)$
the subset consisting of $\t\in\O(S)$ such that every triangle in
$\t$
has at least one edge that is not in ground state. For a triangulation
$\tau\in\O(S)$, we define
%
\begin{equation}
\label{fissi} \Phi(S,\t):= 
\sum_{z\in\midpoints(S)}|
\t_z|;
\end{equation}
note that the boundary edges are omitted from this sum.
We let $Z$ as usual~denote the partition function, that is, $Z=\sum_{\si
\in\O}\lambda^{|\si|}$, and
$Z_S=\break \sum_{\t\in\O(S)}\lambda^{\Phi(S,\t)}$ the restricted partition
function within~$S$.
Clearly, for any ground state region~$S$, we have the bound
\[
Z\geq\lambda^{|\partial S|}Z_S Z_{S^c},
\]
where $|\partial S|$ is the length of the boundary of~$S$. Notice that
if $\cS_x(\si)=S$, then $\si$ contains the edges of the boundary
$\partial S$ and the restriction of $\si$ to $\L(S)$ must be given by
some $\t\in\O^*(S)$. Therefore,
%
\begin{eqnarray}
\label{sxse}
\mu(\cS_x=S) &=&\frac{\sum_{\si}\lambda^{|\si|} \mathbf{1}({\cS_x(\si)=S})}{Z}
\nonumber\\
&\leq&\frac{ \lambda^{|\partial S|}Z_{S^c}\sum_{\t\in\O
^*(S)}\lambda^{\Phi(S,\t)}
}{Z}
\\
&\leq&\frac{\sum_{\t\in\O^*(S)}\lambda^{\Phi(S,\t)}}{Z_S}.\nonumber
\end{eqnarray}
%
Next, we can use the trivial bound $Z_S\geq\lambda^{\Phi(S,\bar\t
)}$, where
$\bar\t$ denotes
the ground state triangulation inside~$S$, and the fact that $\t\in\O
^*(S)$ implies
$\Phi(S,\t) - \Phi(S,\bar\t) \geq\frac{1}2 |S|$ (because each of the
$|S|$ triangles
in~$\tau$ contains at least one edge not in ground state, and each edge
is in two triangles).
Hence, for $\lambda\leq1$,
\[
\mu(\cS_x=S)\leq\sum_{\t\in\O(S)}
\lambda^{({1}/2) |S|}.
\]
Finally, from Anclin's argument~\cite{Anclin} mentioned in
Section~\ref{subsec:related},
there are at most $2^{3|S|/2}$ triangulations
of a region $S$ since the number of interior edges is at most $3|S|/2$.
This proves \eqref{pereaso}.
\end{pf}

We turn now to a second preliminary estimate.
Let $\Phi(S,\si)$ be defined as in \eqref{fissi}, and
consider the quantity
%
\begin{equation}
\label{phix} \phi_x(\si):=\Phi \bigl(\cS_x(\si),\si
\bigr)-\Phi \bigl(\cS_x(\si),\bar\si \bigr),
\end{equation}
measuring the deviation of the total length of the
triangulation~$\sigma$
in the region around~$x$ with respect to the ground state.
The next lemma shows that, for sufficiently small~$\lambda$, $\phi_x$
has exponentially decaying tail probability.

%
\begin{lemma}\label{expdecS}
There exists $\lambda_0\in(0,1)$ such that, for all sets of
constraints, all
points $x\in\L$, all $\lambda>0$
and all $k\in\bbN$,
%
\begin{equation}
\label{expdecS1} \mu \bigl(\phi_x(\sigma)=k \bigr)\leq(\lambda/
\lambda_0)^{k.} 
\end{equation}
\end{lemma}

\begin{pf}
We write
\begin{eqnarray*}
\mu \bigl(\phi_x(\sigma)=k \bigr)&=\sum
_{S} \mu \bigl(\phi_x(\si)=k;
\cS_x(\si)=S \bigr).
\end{eqnarray*}
As in \eqref{sxse}, we have
\begin{eqnarray*}
&&\mu \bigl(\phi_x(\si)=k; \cS_x(\si)=S \bigr) \\
&&\qquad\leq
\frac{\sum_{\t\in\O^*(S)}\lambda^{\Phi(S,\t)}\mathbf{1}({\Phi(S,\t)=\Phi
(S,\bar\t)+k)}}{Z_S}.
\end{eqnarray*}
Next, recall that if $\cS_x(\si)=S$, then $\Phi(S,\t)\geq\Phi
(S,\bar\t
) + \frac{1}2 |S|$.
Thus, using $Z_S\geq\lambda^{\Phi(S,\bar\t)}$, and $k\geq\frac{1}2 |S|$,
it follows that
for any $\lambda_0\in(0,1)$,
%
\begin{equation}
\label{expdec3} \mu \bigl(\phi_x(\sigma)=k \bigr) \leq(\lambda/
\lambda_0)^{k}\sum_{S} \sum
_{\t\in\O(S)} \lambda_0^{|S|/2}.
\end{equation}
As in the proof Lemma~\ref{expdecS0} 
there are at most $2^{3|S|/2}$ triangulations $\t\in\O(S)$.
Moreover, the number of connected ground state regions $S$ containing
the point $x$ as above and with a given size $|S|=\ell$ is bounded by
$5^{2\ell}$. To see this, observe that every such $S$ must be a
collection of $\ell$ connected
ground state triangles, so it suffices to count the number of such
connected sets
that contain a given ground state triangle.
The latter is at most $5^{2\ell}$ since each ground state triangle can
be connected to at most five
others, and a general connected set can be explored by a path of
length~$2\ell$.
The number five appears in the worst case where a ground state triangle
is of minimal length; that is,
it has one horizontal edge, one vertical edge and one unit diagonal
(there are at most two
possible ground state triangles adjacent to each flat edge and one
adjacent to the unit diagonal).

In conclusion, the claim follows by summing in~\eqref{expdec3}:
%
\begin{equation}
\label{expdec33} \mu \bigl(\phi_x(\sigma)=k \bigr) \leq(\lambda/
\lambda_0)^{k}\sum_{S}
2^{3|S|/2}\lambda _0^{|S|/2}\leq(\lambda/
\lambda_0)^{k}\sum_{\ell=2}^\infty
C_0^\ell \lambda_0^{\ell/2},
\end{equation}
where $C_0=5^22^{3/2}$, and we use the fact that the minimal
nontrivial $S$ has size at least $2$.
It suffices to take $\lambda_0>0$ such that (e.g.) $\lambda
_0^{1/2}C_0\leq1/2$
to ensure that the last summation is less than~$1$.
\end{pf}

We are now able to prove the exponential tail bound advertised at the
start of this
section. Notice that if $\si\in\O$ and $x\in\midpoints$ are such that
$|\si_x|= |\bar\si_x|+k$,
then necessarily $\phi_x(\si)\geq k$. Hence from Lemma~\ref{expdecS} we
immediately
have:

%
\begin{corollary}\label{cor1}
Let $\lambda_0$ be as in Lemma~\ref{expdecS}. For every point $x\in
\L$ and
every $k\in\bbN$,
%
\begin{equation}
\label{expdeco} \mu\bigl(|\si_x|= |\bar\si_x|+k\bigr) \leq(
\lambda/\lambda_0)^{k}.
\end{equation}
\end{corollary}

%
Recall that $\cS=\cS(\si)$ is the random subset of the plane defined as
the union
of all B-triangles in $\si$.
Our next result
establishes that for small~$\lambda$ the probability
that $\cS$ contains any particular region~$V$ decays exponentially in the
size of~$V$.
Since $\cS$ is a ground state region, we naturally consider only
ground state
regions~$V$.
We write $|V|$ for the number of triangles in $V$; since all triangles
are unimodular,
this is always equal to twice the area of~$V$.

\begin{lemma}\label{cor2}
Let $\lambda_0\in(0,1)$ be as in Lemma~\ref{expdecS}. For all sets of
constraints, all ground state regions $V$ w.r.t. the constraints and
all $\lambda>0$,
%
\begin{equation}
\label{expdeco1} \mu(V\subseteq\cS)\leq(4\lambda/\lambda_0)^{|V|/2}.
\end{equation}
\end{lemma}

\begin{pf}
Suppose first that $V$ is connected. Let $T$ be a ground state triangle
in~$V$, and let $\cS_0$ denote the connected component of $\cS$
containing $T$. Then
\[
\mu (V\subseteq\cS )\leq\sum_{S: S\supseteq V}\mu (\cS
_0=S ),
\]
where the sum is over all connected ground state regions $S$ such that
$S\supseteq V$.
Lemma~\ref{expdecS0} implies that
%
\begin{equation}
\label{pereaso1} \mu (\cS_0=S )\leq \bigl(2^{3}\lambda
\bigr)^{|S|/2}.
\end{equation}
Summing over all $S$ as above and estimating by $5^{2\ell}$ the number
of connected $S\ni T$ with $|S|=\ell$, with the notation of \eqref{expdec33}
we have
\[
\mu (V\subseteq\cS )\leq\sum_{\ell\geq|V|}
C_0^{\ell
} \lambda ^{\ell/2}\leq(\lambda/
\lambda_0)^{|V|/2}.
\]
%

We now turn to the case where $V$ has several connected components, say
$V_1,\ldots,V_n$.
Let $T_1,\ldots,T_n$ denote fixed ground state triangles such that
$T_i\in V_i$, and let $\cS_i$, $i=1,\ldots, n$
denote the random sets defined as the
connected components of $\cS$ such that $\cS_i\ni T_i$. Notice that the
$\cS_i$ need not be distinct, since a single connected component of
$\cS
$ may contain more than one of the triangles $T_i$. However, the event
$V\subseteq\cS$ implies that there exist indices
$1\leq i_1 <\cdots<i_m\leq n$, $1\leq m\leq n$, and disjoint connected
components
$S_{i_1},\ldots,S_{i_m}$ with $|S_{i_k}|=\ell_{k}$ satisfying $\sum_{k=1}^m\ell_{k}\geq|V|$, and such that $S_{i_k}\ni T_{i_k}$ and $\cS
_{i_k}=S_{i_k}$, for all $k=1,\ldots,m$.

Now, for any $1\leq m\leq n$, for any choice of $ i_1,\ldots,i_m$ as
above and any choice of disjoint ground state regions $S_{i_1},\ldots,S_{i_m}$
such that $S_{i_k}\ni T_{i_k}$, repeating the argument of Lemma~\ref
{expdecS0}, one has
%
\[
\mu (\cS_{i_k}=S_{i_k}, \forall k=1,\ldots,m )\leq (8
\lambda)^{\sum
_{k=1}^m|S_{i_k}|/2}.
\]
The number of choices of $i_1,\ldots,i_m$ is bounded by ${n}\choose m$.
Therefore, with the notation of \eqref{expdec33}, using $C_0\lambda
_0^{1/2}\leq1/2$, $n\leq|V|$ and a union bound we obtain
\begin{eqnarray*}
\mu (V\subseteq\cS )&\leq& \sum_{m=1}^n
\pmatrix{n
\cr
m}\sum_{\ell_1\geq2}\cdots\sum
_{\ell
_m\geq2} 1 \Biggl(\sum_{k=1}^m
\ell_k\geq|V| \Biggr) \bigl(C_0\lambda^{1/2}
\bigr)^{\sum
_{k=1}^m\ell_k}
\\
& \leq&\sum_{m=1}^n\pmatrix{n
\cr
m} (
\lambda/\lambda_0)^{|V|/2} \leq2^n(\lambda/
\lambda_0)^{|V|/2}\leq(4\lambda/\lambda_0)^{|V|/2}.
\end{eqnarray*}
%
\upqed\end{pf}

Finally, we use Lemma~\ref{cor2} to derive the spatial mixing estimate
stated in Theorem~\ref{thm:zero} in the \hyperref[sec:intro]{Introduction}.
This bound will be expressed in terms of a natural distance between
the edge~$\sigma_x$ and any subset of midpoints $A\subseteq\L$, which
we now define.

Let $\si_x$ be an edge consistent with the constraints $(\eta,\D)$.
Also, let $\bar\sigma(\sigma_x)$ denote the new ground
state triangulation when edge $\sigma_x$ is added to the constraints.
Let $\mu^{\si_x}$ denote the Gibbs distribution on triangulations with
this extra constraint, that is,
$\mu^{\si_x}$ is the
probability $\mu$ conditioned on the event that the edge at $x$ equals~$\si_x$.
Intuitively, the distributions $\mu^{\si_x}$ and $\mu$
should have very similar marginals at midpoints~$z$ that are far from
the influence
region $\Upsilon(\si_x)$, as defined in Section~\ref{subsec:upsilon}.
To quantify this we introduce the following distance. Consider the
graph $\cG(\si_x)$
whose vertices are the triangles in $\bar\sigma(\sigma_x)$, where two
triangles are neighbors if and only if they share an edge.
Let $d(A,\si_x)$ denote the graph distance in $\cG(\si_x)$ between the
influence region $\Upsilon(\si_x)$ and the set of triangles spanned by
edges $z\in A$ [i.e.,
the union of all triangles in $\bar\sigma(\sigma_x)$
containing the edges $\bar\sigma_z(\sigma_x)$, $z\in A$], where the
distance between two sets of vertices of a graph
is interpreted as the minimum distance between a pair of vertices, one
from each set. In words, $ d(A,\si_x)$ is the minimal number of
triangles in
$\bar\sigma(\sigma_x)$ needed to connect an edge in~$A$
to the influence region~$\Upsilon(\sigma_x)$. Below we use $\|\mu
^{\sigma_x}-\mu\|_A$ for
the variation distance between the distributions $\mu^{\si_x}$ and
$\mu
$ at $A$, that is, the variation distance between
the marginals on edges $z\in A$. We write $|A|$ for the cardinality of $A$.

%
\begin{theorem}\label{covdecga}
There exists $\lambda_1\in(0,1)$ such that, for all sets of
constraints, all
$x\in\L$ and $A\subseteq\L$
and all consistent values of the edge~$\sigma_x$, we have
%
\begin{equation}
\label{covdecga1} \bigl\|\mu^{\sigma_x}-\mu\bigr\|_A \leq 
|A|(
\lambda/\lambda_1)^{d(A,\sigma_x)/8}.
\end{equation}
\end{theorem}

\begin{pf}
Observe that apart from the set $\Upsilon(\si_x)$ the ground states for
$\mu^{\sigma_x}$ and $\mu$ coincide. When we talk about triangles below
we always refer to the triangles in the common ground state
triangulation [outside the region $\Upsilon(\si_x)$]. Moreover, we
identify $A$ with the union of all triangles spanned by edges from $A$.
Recall that the common ground state triangulation is unique up to unit
diagonal flips and that a unit diagonal flip can only happen when there
is a unit square in the ground state centered at that midpoint.
Consider the (unique) planar graph $\G$ obtained from the common ground
state edges by deleting all flippable unit diagonals. In this graph
each face is either a triangle or a unit square.
Let $\bbP$ denote the independent coupling of $\mu^{\sigma_x},\mu$ and
let $\si,\si'$ denote the associated random triangulations. The pair
$\si,\si'$ induces a coloring of the faces of $\G$ as follows. If the
face $F$ in $\G$ is a triangle, then let $F$ be black if both $\si,\si'$ have that triangle, and let $F$ be white otherwise. If $F$ is a unit
square, then let $F$ be black if both $\si,\si'$ have that unit square
(with possibly opposite unit diagonals), and let $F$ be white otherwise.
Two faces are $*$-connected if they share a vertex, while they are
connected if they share an edge.
Consider the event $\cE$ that there exists a $*$-connected chain $\cC$
of black faces in $\G$ that separates $\Upsilon(\si_x)$ from $A$.
By definition, any such chain $\cC$ divides the system into three
ground state regions: one containing $ \Upsilon(\si_x)$, another
containing $A$ and one consisting of $\cC$ itself. Denote the first
region by~$\cC_1$. We say that $\cC$ is smaller than $\cC'$ if $\cC
_1\subset\cC'_1$. Then, on the event $\cE$, there is a smallest
separating chain as above; call it $\cC^*$. By conditioning on the
value of $\cC^*$, an application of
the Markov property for the Gibbs distributions $\mu^{\sigma_x},\mu$
implies that
these measures can be coupled in such a way that the two configurations
agree on $A$. Therefore,
%
\begin{equation}
\label{covdecga10} \bigl\|\mu^{\sigma_x}-\mu\bigr\|_A \leq\bbP \bigl(
\cE^c \bigr).
\end{equation}
Next, observe that the complementary event $\cE^c$ implies that there
exists a connected chain of white faces in $\G$
that connects $A$ and $ \Upsilon(\si_x)$. Note that if a white face $F$
is a triangle, then $F$ belongs to either $\cS(\si)$ or $\cS(\si')$,
while if $F$ is a unit square, then either $\cS(\si)$ or $\cS(\si')$
contains a triangle $T$ within $F$. ($T$ is one of the four possible
triangles one can inscribe in $F$.)
It follows that on the event $\cE^c$ there exists a connected chain of
faces $\cD$ in $\G$ that
connects~$\Y(\si_x)$ and~$A$, and a ground state region~$\bar\cD$
within $\cD$ with $|\bar\cD|\geq|\cD
|$ such that
the union of the regions
$\cS(\si)$ and $\cS(\si')$ includes $\bar\cD$. Here $ |\cD|$ is the
number of faces in $\cD$, while $|\bar\cD|$ denotes the number of
ground state triangles
in~$\bar\cD$.
Plainly either $\cS(\si)$ or $\cS(\si')$ contains at least
$|\bar\cD|/2\geq|\cD|/2$ ground state triangles.
Also, by definition $|\cD|\geq d(A,\sigma_x)/2$.
So a union bound yields
\[
\label{covdecga11} \bigl\|\mu^{\sigma_x}-\mu\bigr\|_A \leq\sum
_\cD\sum_{V\subseteq\cD: |V|\geq|\cD|/2} \bigl[ \mu
^{\sigma
_x}(V\subseteq\cS) + \mu(V\subseteq\cS) \bigr],
\]
where the first sum ranges over all connected chains of faces in $\G$
touching~$A$
and such that $|\cD|\geq d(A,\sigma_x)/2$, while the second sum is over
all ground state regions $V$ within $\cD$
with $|V|\geq|\cD|/2$.
From Lemma~\ref{cor2}
it follows that for any ground state region~$V$, one has
$\mu^{\sigma_x}(V\subseteq\cS)\leq(4\lambda/\lambda
_0)^{|V|/2}$ and $ \mu
(V\subseteq\cS)\leq(4\lambda/\lambda_0)^{|V|/2}$,
and therefore
\[
\sum_{V\subseteq\cD\dvtx  |V|\geq|\cD|/2} \bigl[ \mu^{\sigma
_x}(V\subseteq
\cS) + \mu(V\subseteq\cS) \bigr]\leq 8^{|\cD|}(4\lambda/
\lambda_0)^{|\cD|/4}.
\]
%
Here we have bounded by $8^k$ the number of all possible ground state
regions $V$ within a given collection of
$k$ faces in $\G$ (the worst case being when all faces are unit squares).
As in the proof of Lemma~\ref{expdecS}, the number of connected $\cD$
touching a given triangle $T\subseteq A$ with $|\cD|=k$ is at most
$C^{k}$, so that
\[
\label{covdecga12} \bigl\|\mu^{\sigma_x}-\mu\bigr\|_A \leq|A| \sum
_{ k\geq d(A,\sigma_x)/2}
C_1^{k}(\lambda/
\lambda_0)^{k/4} \leq|A|(\lambda/\lambda_1)^{d(A,\si_x)/8},
\]
for suitable constants $C,C_1,\lambda_1>0$.
\end{pf}

%
%
\begin{remark}
The above theorem can be extended without substantial effort to the
case when $\s_x$ is replaced by a set of extra constraints $(\s
_{x_1},\ldots,\s_{x_k})$ provided one replaces the distance $d(A,\s_x)$
with a suitable new distance defined in terms of the influence regions
of the edges $(\s_{x_1},\ldots,\s_{x_k})$.
\end{remark}

%
\section{Upper bounds on the mixing time for \texorpdfstring{$\lambda<1$}{$lambda<1$}}\label{chain}
We consider the Glauber dynamics on triangulations of the region
$\Lambda^0:=\Lambda^0_{m,n}$, as defined
in Section~\ref{subsec:glauber}. We allow an arbitrary set of
constraints $(\eta,\Delta)$, so that
our results apply in particular to triangulations of any lattice
polygon. Our main result in this
section is the following polynomial upper bound on the mixing time for
all sufficiently
small~$\lambda$, which was stated as Theorem~\ref{thm:one} in the
\hyperref[sec:intro]{Introduction}.
[Recall the definition of mixing time from equation~(\ref{eq:tmix}).]

%
\begin{theorem}\label{gapknbound}
There exist constants $C>0$ and $\lambda_0\in(0,1)$ such that, for
all $\lambda
\leq\lambda_0$, all $m,n\in\bbN$, and all constraints $(\eta,\D)$,
%
\begin{equation}
\label{gapknbound1} T_{\mathrm{ mix}}\leq Cmn (m+n).
\end{equation}
\end{theorem}

%
We first establish a key coupling estimate and then turn to the proof
of Theorem~\ref{gapknbound}.
\subsection{Path coupling}\label{sec:pathcoupling}
Let $\si$ and $\t$ be two triangulations so that $\si$ differs from
$\t
$ only on the edge with midpoint~$x$.
Let $\alpha>1$ be a fixed positive number. If $\si_x$ and $\t_x$ are
opposite unit diagonals, define $\Delta(\si,\t)=\a^2-1$; 
otherwise, let
%
\begin{equation}
\label{def:d} \Delta(\si,\t) = \bigl|\alpha^{|\si_x|}-\alpha^{|\t_x|} \bigr|.
\end{equation}
Note that $\Delta$ is defined only over pairs of triangulations that
are adjacent in the flip graph.
We extend $\Delta$ to all pairs of triangulations by defining $\Delta
(\sigma,\tau)$ to be
the shortest path distance between~$\sigma$ and~$\tau$ in the flip
graph, with edge
lengths given by the above $\Delta$ values.


\begin{lemma}\label{lem:pathcoupling}
There exists $\a_0>1$ such that the following holds for all $\a\geq
\a
_0$, and $\lambda\a\leq1$.
Let $\si,\t$ be two triangulations that differ in exactly one edge,
and let $\si',\t'$ denote the random triangulation obtained after one
step of the Markov chain.
Then there exists a coupling of $\si'$ and $\t'$ so that their
expected distance satisfies
%
\begin{equation}
\label{bound:d} \bbE \bigl[\Delta \bigl(\si',\t' \bigr)
\bigr] \leq\Delta(\si,\t) \biggl(1-\frac
{1}{2|\midpoints|} \biggr),
\end{equation}
where $\midpoints$ is the set of midpoints of nonconstraint edges.
\end{lemma}

\begin{pf}
Let $x$ be the midpoint so that $\si_x$ and $\t_x$ differ.
Assume without loss of generality that $|\si_x|\geq|\t_x|$.
Since $\si$ differs from~$\t$ by one flip at $x$, the edges $\si_x$
and $\t_x$
must be the two diagonals of the minimal parallelogram of~$\si_x$.
The coupling is such that the random midpoint $Y\in\L$ chosen
is the same for both triangulations. Thus if $Y=x$ then $\si'_x=\tau
'_x$ and
$\Delta(\si',\t')=0$ with probability one.
If instead $Y\neq x$, which happens with probability $(1-1/|\L|)$, then
$\Delta(\sigma',\tau')=\Delta(\sigma,\tau)$ unless $Y$ is one of the
edges of
the minimal parallelogram of~$\sigma_x$ (since only these edges share
a triangle with the
differing edge at~$x$); so it remains only to bound the expected
change in
distance induced by flips at the edges of the minimal parallelogram.

Assume first that $|\si_x|> |\t_x|$. By Proposition~\ref{prop:closest}
and its proof, we are in the situation described in Figure~\ref
{fig:excluded}, with $\si_x=p_1p_2$ and $\t_x=p_3p_4$. As in that
picture, one has that $p_3$ and $p_4$ lie between $p_1$ and $p_2$ both
vertically and horizontally. In this case the $\ell_1$-lengths of $\si
_x$ and $\t_x$ can be written as $|\si_x|=f_1+f_2$, and $|\t
_x|=f_1-f_2$, where $f_1$ is the $\ell_1$-length of $p_1p_4$ and
$f_2$ is the $\ell_1$-length of $p_2p_4$, and we are assuming
w.l.o.g. that $f_1\ge f_2$.
Let $y_1,y_2,z_1,z_2\in\midpoints$ be the midpoints of the edges of
the minimal
parallelogram, so that
$|\si_{y_1}|=|\si_{y_2}|=f_1\geq f_2=|\si_{z_1}|=|\si_{z_2}|$.
We start with the case $Y=z_1$.
Note that if $\si_{z_1}$ (resp., $\t_{z_1}$) is flippable, then $\t
_{z_1}$ (resp., $\si_{z_1}$) is not flippable.
Since $f_1\geq f_2$, if $\si_{z_1}$ is flippable it must flip to an
edge $\si'_{z_1}$ of length $|\si_x|+|\si_{y_1}|=2f_1+f_2$.
Then, for $\a>1,\lambda<1$ with $\a\lambda\leq1$,
the expected increase in distance in this case is bounded by
\[
\bigl(\alpha^{2f_1+f_2}-\alpha^{f_2} \bigr)\frac{\lambda
^{2f_1+f_2}}{(\lambda^{2f_1+f_2}+\lambda^{f_2})} =
\frac{\alpha^{f_2}(1-\alpha^{-2f_1})(\lambda\alpha
)^{2f_1}}{(\lambda
^{2f_1}+1)} \leq\alpha^{f_2}.
\]
If instead $\t_{z_1}$ is flippable, then it flips to an edge of length
$|\t_x|+|\t_{y_1}|=2f_1-f_2$,
resulting in an expected increase in distance of at most
\[
\bigl(\alpha^{2f_1-f_2}-\alpha^{f_2} \bigr)\frac{\lambda
^{2f_1-f_2}}{(\lambda^{2f_1-f_2}+\lambda^{f_2})} \leq
\alpha^{f_2}(\lambda\alpha)^{2f_1-2f_2} \leq\alpha^{f_2}.
\]
Now we turn to the case $Y=y_1$. Note that $\si_{y_1}$ is either
unflippable or
flippable to a longer edge, while $\t_{y_1}$ is either unflippable or
flippable to a shorter edge.
If $\si_{y_1}$ is flippable, it flips to an edge of length $|\si
_{x}|+|\si_{z_1}|=f_1+2f_2$,
and the expected increase in distance is at most
\[
\bigl(\alpha^{f_1+2f_2}-\alpha^{f_1} \bigr)\frac{\lambda
^{f_1+2f_2}}{(\lambda^{f_1+2f_2}+\lambda^{f_1})} \leq
\alpha^{f_1}(\lambda\alpha)^{2f_2} \leq\alpha^{f_1}.
\]
If $\t_{y_1}$ is flippable, it flips to an edge of length $|\t_x|+|\t
_{z_1}|=|f_1-2f_2|$, which causes an expected increase in distance of
at most
\[
\bigl(\alpha^{|f_1-2f_2|}-\alpha^{f_1} \bigr) \leq
\alpha^{|f_1-2f_2|} \leq\alpha^{f_1}.
\]
Summing up all these contributions, we obtain
\[
\bbE \bigl[\Delta \bigl(\si',\t' \bigr) \bigr] \leq
\Delta(\si,\t) \biggl(1-\frac{1}{|\L|} \biggr) 
+\frac{2\alpha^{f_2}}{|\midpoints|}
+\frac{2\alpha^{f_1}}{|\midpoints|}. 
\]
This implies the desired bound in the lemma, since if $\a$ is large enough
($\alpha>5$ suffices), we have
\[
2\a^{f_1}+2\a^{f_2}\leq\tfrac{1}2 \bigl(
\a^{f_1+f_2}-\a^{f_1-f_2} \bigr) = \tfrac{1}2 \Delta(\si,\t),
\]
for any $f_1> f_2\geq1$.

It remains to consider the case $|\si_x|= |\t_x|$; note that here $\D
(\si,\t)=\a^2-1$.
This case can only happen if $\si_x,\t_x$ are opposite unit diagonals,
so $|\si_x|= |\t_x|=f_1+f_2$
with $f_1=f_2=1$. Therefore, the total contribution from the edges of
the minimal parallelogram
(now a square) is bounded by
\[
4 \bigl(\alpha^{3}-\alpha \bigr) \frac{\lambda^{3}}{(\lambda
^{3}+\lambda
)}\leq4
\a^{-1} \bigl(\a^2-1 \bigr) (\lambda\a)^2\leq4
\a^{-1}\D(\si,\t).
\]
This gives the following bound on the expected distance after one flip:
\[
\bbE \bigl[\Delta \bigl(\si',\t' \bigr) \bigr] \leq
\Delta( \si,\t) \biggl(1-\frac{1}{|\L|} \biggr) 
+\frac{4\alpha^{-1}\Delta(\si,\t)}{|\midpoints|}.
\]
The desired bound follows by taking $\a\geq8$.
\end{pf}

\textit{Note}: In the above proof it suffices to take $\alpha=8$ and
$\lambda\leq\alpha^{-1}=1/8$.
Hence Theorem~\ref{gapknbound} holds with $\lambda_0=1/8$.

\subsection{Proof of Theorem~\texorpdfstring{\protect\ref{gapknbound}}{5.1}}\label{subsec:path}
Once the estimate in Lemma~\ref{lem:pathcoupling} is available the argument
is rather standard; see, for example,
\cite{LPW}, Theorem~14.6. Indeed, from the triangle inequality and the
definition
of the metric $\Delta(\cdot,\cdot)$, it follows that estimate~\eqref
{bound:d} can
be extended to \textit{any} pair of triangulations $\si,\t$ under a
suitable coupling.
This fact, together with the Markov property yields the bound
%
\begin{equation}
\label{bound:d1} \bbE \bigl[\Delta \bigl(\si^{(k)},\t^{(k)}
\bigr) \bigr] \leq\Delta(\si,\t) \biggl(1-\frac
{1}{2|\midpoints|} \biggr)^k
\leq \Delta(\si,\t) e^{-k/2|\L|},
\end{equation}
where $\bbE$ denotes expectation over the coupling of the random
triangulations $\si^{(k)},\t^{(k)}$ obtained after $k$ steps of the
Markov chain started at $\si,\t$, respectively.
The diameter bound in Lemma~\ref{lem:connected}, and the fact that the distance
between any two adjacent triangulations is at most $\alpha^{m+n}$, imply
that $\Delta(\si,\t)\le e^{C_0(m+n)}$ for some constant $C_0>0$.
Using also the fact that $\Delta(\si,\t)\geq c$ if $\si\neq\t$, for
some constant $c=c(\a)>0$,
an application of Markov's inequality yields
\[
\bigl\|p^k(\si,\cdot)-p^k(\t,\cdot)\bigr\|\leq
c^{-1}e^{C_0(m+n)- k/2|\L|},
\]
uniformly in the initial conditions $\si,\t$. Taking, for example,
$k=4|\L|C_0(m+n)$ completes the proof. 

\textit{Note}: Observe from the above proof that Theorem~\ref{gapknbound}
actually holds
with the bound $T_{\mathrm{ mix}}\leq C|\L| (m+n)$, where $|\L|$ is the
number of \textit{nonconstraint}
edges. This bound may be much better in cases where there are many constraints.

%
%
%


\subsection{The \texorpdfstring{$1\times n$}{$1 times n$} case for all \texorpdfstring{$\lambda<1$}{$lambda<1$}}
Here we prove that, in the special case of the 1-dimensional region
$\Lambda^0_{1,n}$,
the bound of Theorem~\ref{gapknbound} holds for \textit{all}
$\lambda<1$.

%
\begin{theorem}\label{gapknbound1d}
Set $m=1$. For any $\lambda<1$, there exists a constant $C>0$ such
that for
all $n\in\bbN$, and for all constraints $(\eta,\D)$,
%
\begin{equation}
\label{gapknbound11d} T_{\mathrm {mix}}\leq Cn^2.
\end{equation}
\end{theorem}

To prove the theorem, we show that the path coupling argument of
Section~\ref{sec:pathcoupling} works for all $\lambda<1$ in this case.
For 1D triangulations we may define the length of an edge as its
horizontal length. Fix $\a>1$.
Let $\si,\t$ be two triangulations that differ at exactly one edge~$x$
with $|\si_x|\geq|\tau_x|$.
Then either the edges at $x$ have equal length, in which case they are
necessarily both unit diagonals,
or they have lengths $\ell-1$ (in~$\t$) and $\ell+1$ (in $\si$), for
some integer $\ell\geq1$.
In the first case we set $\D(\si,\tau)=\a^2-1$; in the second case we
set $\D(\si,\tau)=\a^{\ell+1}(1-\a^{-2})$.
As before, we extend~$\Delta$ to the shortest-path metric on all pairs
in~$\O$.
As detailed in Section~\ref{subsec:path},
estimate \eqref{gapknbound11d} will follow from the lemma below and
the fact that the maximal distance between two arbitrary configurations
is $e^{O(n)}$.

\begin{lemma}\label{1n}
For any $\lambda<1$ there exist constants $\a>1$, $\delta>0$ such
that the
following holds.
Let $\si,\t$ be two triangulations that differ at exactly one edge,
and let $\si',\t'$ denote the random triangulation obtained after one step.
There exists a coupling of $\si',\t'$ such that
%
\begin{equation}
\label{prodo1} \bbE \bigl[\Delta \bigl(\si',\t' \bigr)
\bigr]\leq(1-\delta/n)\Delta(\si,\t).
\end{equation}
\end{lemma}

\begin{pf}
Say that the initial discrepancy is at midpoint $x$. In the coupling,
we pick the same random midpoint
$X$ to be updated, so that $\bbE[d(\si',\t')]$ is given by
\[
\bbE \bigl[\Delta \bigl(\si',\t' \bigr) \bigr]=(1-1/n)
\Delta(\si,\t)+\frac{1}n U(\si,\t),
\]
where $U(\si,\t)$ is the term coming from possible new discrepancies at
$x\pm1/2$.
To compute the latter, we
start with the case where $\si_x\neq\t_x$ are both unit diagonals, so
that $\Delta(\si,\t)=\a^2-1$.
In this case, the edges at $x\pm1/2$ are necessarily vertical in both
$\si,\t$, and we can create a discrepancy at either $x\pm1/2$ by
letting the edge there increase to length $2$. Only one of the two
triangulations can flip at either $x\pm1/2$. Thus, taking
expectations, in this case we have
\[
U(\si,\t)\le\frac{2\lambda^2}{1+\lambda^2} \bigl(\a^2-1 \bigr) =
\frac
{2\lambda^2}{1+\lambda^2} \Delta (\si,\t).
\]
Next, consider the remaining case where $|\si_x|=\ell+1$ and $|\t
_x|=\ell-1$. Necessarily the edges at $x\pm1/2$ have length $\ell$ in
both $\si,\t$.
Again, only one of the two triangulations can flip at either $x\pm
1/2$. In the worst case scenario, $x\pm1/2$ can either both increase
to $\ell+2$, or both decrease to $\ell-2$, or one can increase to
$\ell
+2$ and the other decrease to $\ell-2$. These produce respectively the terms
\begin{eqnarray*}
U(\si,\t)&=&\frac{2\lambda^2}{1+\lambda^2}\a^{\ell+2} \bigl(1-\a ^{-2} \bigr)
= \frac{2\a\lambda
^2}{1+\lambda^2} \Delta(\si,\t);
\\
U(\si,\t)&=&\frac{2}{1+\lambda^2}\a^{\ell-1} \bigl(1-\a^{-2} \bigr) =
\frac
{2}{\a(1+\lambda
^2)} \Delta(\si,\t);
\\
U(\si,\t)&=&
 \biggl(\frac{\a\lambda^2}{1+\lambda^2}+\frac{1}{\a(1+\lambda
^2)}
\biggr) \Delta(\si,\t).
\end{eqnarray*}
Thus it suffices to show that, for any $\lambda<1$, we can find $\a
>1$ and
$\delta>0$ such that
\[
\max \biggl\{ \frac{2\a\lambda^2}{1+\lambda^2},\frac{2}{\a
(1+\lambda^2)} \biggr\} \leq 1-\delta.
\]
This is easily verified.
\end{pf}

\textit{Note}: The bound in Theorem~\ref{gapknbound1d} should be compared
with the known bound $\tmix=\Theta(n^3\log n)$ in the case of 1D
triangulations with $\lambda=1$ (and no constraints). This follows
from a
one-to-one correspondence between 1D triangulations and lattice paths,
and the
results of Wilson~\cite{Wilson} on the mixing time of the flip dynamics
for such paths.

\section{Lower bounds on the mixing time for \texorpdfstring{$\lambda>1$}{$lambda>1$}}\label{sec:lowerB}
As a general rule, our lower bounds on the mixing time are obtained by
exploiting the
fact that it takes a long time to change the orientation of an
initially long edge.
Below, we say that an edge of a triangulation is \textit{positively
oriented} (resp.,
\textit{negatively oriented}) if its leftmost endpoint lies below
(resp., above) its
rightmost endpoint. A vertical or horizontal edge is neither positively
nor negatively oriented.


%

Our first result is an exponential lower bound on the mixing time (in
the absence of constraints)
that holds for any $\lambda>1$. This proves Theorem~\ref{thm:two} in the
\hyperref[sec:intro]{Introduction}.

%
\begin{theorem}\label{thm:alla1}
For every $\lambda>1$, there exists $c>0$ such that the mixing time
of the
Glauber dynamics
on triangulations of $\Lambda^0_{m,n}$ in the absence of constraints
satisfies $\tmix\geq\exp{(c (m+n))}$ for all $m,n\geq1$.
\end{theorem}

\begin{pf}
For any $A\subseteq\O$ we write
\[
Z(A)=\sum_{\si\in A}\lambda^{|\si|}
\]
for the partition function restricted to $A$. If $\partial A$ denotes
the set of $\s\in A$ that are adjacent to $A^c$ in the flip graph, then
the standard conductance bound (see, e.g., \cite{LPW}, Theorem~7.3)
will allow us to deduce the theorem once we find a set $A$ such that
$\mu(A)\leq1/2$ and
%
\begin{equation}
\label{aset} \frac{Z(\partial A)}{Z(A)}\leq e^{-c (m+n)}.
\end{equation}
Assume w.l.o.g. that the horizontal coordinate $n\geq m$ (else the
same argument applies
with the roles of $n$ and~$m$ reversed). We define $A\subseteq\O$ as
the set of
triangulations $\si\in\O$ such that every internal midpoint with
half-integer\vspace*{1pt} vertical coordinate $v$
is not positively (resp., not negatively) oriented if $v+\frac{1}2$ is
odd (resp., even). Notice that any $\si\in A$ has the ``herringbone''
structure illustrated in Figure~\ref{fig:herr}. In particular, any
triangulation
in $A$ consists of 1D triangulations in each horizontal layer. Moreover,
if $\si\in A$, then all edges with integer vertical coordinate are
horizontal and frozen (unflippable).

Now, the total length of any $\si\in A$ can be written as
\[
|\si| = (m+1)n + (2n+1)m + \sum_{j=1}^m
L \bigl(\si(j) \bigr),
\]
where $(m+1)n$ is the total length of all horizontal (unit length)
edges, $(2n+1)m$ is the sum of all vertical lengths of the edges of
$\si
$ and $L(\si(j))$ denotes the total horizontal length of the internal
edges of the 1D triangulation $\si(j)$ in the $j$th layer.

\begin{figure}

\includegraphics{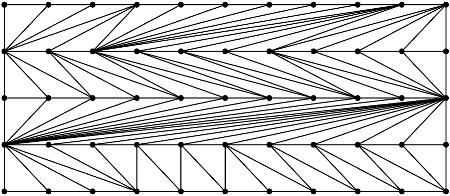}

\caption{Example of a triangulation of a $4\times10$ region, with the
``herringbone'' structure
of the set $A$ introduced in the proof of Theorem~\protect\ref{thm:alla1}.}
\label{fig:herr}
\end{figure}

Below, we let $\O^+_{1,n}$ denote the set of all 1D triangulations such
that all internal edges are not negatively oriented.
From the well-known map between lattice paths and 1D triangulations
(see, e.g.,~\cite{KZ}), one has that
$\O^+_{1,n}$ is in one-to-one correspondence with lattice paths
(integer valued sequences with $\pm1$ increments) of length $2n$,
which start at zero, end at zero and stay nonnegative. Moreover, under
this correspondence, the total length $L(\pi)$ of $\pi\in\O
^+_{1,n}$ is
precisely the area under the path $\pi$. The maximal element in $\O
^+_{1,n}$, denoted $\hat\pi$, has total length
\[
L(\hat\pi)=\sum_{i=1}^ni+ \sum
_{i=1}^{n-1}i = n^2.
\]
Clearly, $\si\in\partial A$ implies that there exists a layer $j$ such
that $\si(j)$ has an internal vertical edge, that is, an edge with
horizontal length zero. In particular, on that layer we must have
$L(\si
(j))\leq n^2 - n$.
Therefore, 
writing $Z(A),Z(\partial A)$ as products over layers, and summing over
the $m$ choices of the distinguished layer $j$ with $L(\si(j))\leq n^2
- n$, one finds
%
\begin{eqnarray*}
\label{aset1} \frac{Z(\partial A)}{Z(A)}&\leq& m\frac{\sum_{\pi\in\O^+_{1,n}\dvtx L(\pi)\leq n^2 - n}\lambda
^{L(\pi)}}{\sum_{\pi
'\in\O^+_{1,n}}\lambda^{L(\pi')}} 
\\
&\leq& m\sum_{\pi\in\O^+_{1,n}\dvtx L(\pi)\leq n^2 - n}\lambda ^{L(\pi
)-n^2}\leq m\sum
_{k\geq n}p_n(k)\lambda^{-k},
\end{eqnarray*}
where we have estimated $\sum_{\pi'\in\O^+_{1,n}}\lambda^{L(\pi
')}\geq\lambda
^{L(\hat\pi)}= \lambda^{n^2}$, and we use $p_n(k)$ to denote the
number of
$\pi\in\O^+_{1,n}$ which have $L(\pi)= n^2 - k$. From the
correspondence with lattice paths it is seen that $p_n(k)$ is bounded
by the number $p(k)$ of all partitions of the integer $k$. 
To see this, observe that: (i)~if $\pi\in\O^+_{1,n}$ with $L(\pi)= n^2
- k$, then the region with area $L(\hat\pi)- L(\pi)$ obtained by
``subtracting'' $\pi$ from $\hat\pi$ is a Young diagram with $k$
boxes; and
(ii)~any Young diagram with $k$ boxes can be seen uniquely as a region
obtained by subtraction as above for some (not necessarily
nonnegative) lattice path $\pi$ with area $n^2-k$. Since $p(k)$ is the
number of Young diagrams with $k$ boxes, restricting to nonnegative
lattice paths $\pi$ one
has $p(k)\geq p_n(k)$.
Finally, using the well-known fact that $p(k)=\exp({O(\sqrt k)})$, we
arrive at the desired bound \eqref{aset}.
\end{pf}

The lower bound in Theorem~\ref{thm:alla1} can be improved as follows
in the case $m\ll\sqrt n$.

%
\begin{theorem}\label{lbgap}
For every $\lambda>1$, there exists $c>0$ such that the mixing time
of the
Glauber dynamics
on triangulations of $\Lambda^0_{m,n}$, in the absence of constraints,
satisfies $\tmix\geq\exp{(c n^2/m)}$ for all $n\geq m^2/c$, $m\geq1$.
\end{theorem}

\begin{pf}
Take $\e\in(0,1)$ and divide the midpoints in $\L:=\Lmn$ into three
regions as follows:
$\L_\ell$ is the set of $x\in\L$ with horizontal coordinate between $0$
and $\e n$,
$\L_c$ is the set of $x\in\L$ with horizontal coordinate between $\e n$
and $(1-\e) n$,
and $\L_r$ is the remainder of~$\L$.
Consider the set $\L_1\subseteq\L_c$ of all internal midpoints $x\in
\L_c$
whose vertical coordinate is half-integer.
Let $A\subseteq\O$ denote the set of triangulations~$\sigma$ such
that every
edge~$\sigma_x$ with $x\in\midpoints_1$ with vertical coordinate~$v$
is positively or negatively oriented according to the parity of
$v+\frac{1}2$.
In particular, adjacent layers are required to have
one-dimensional edges of opposite orientation throughout the whole
central region.
This is a relaxed version of the set $A$ appearing in the proof of
Theorem~\ref{thm:alla1}
(see Figure~\ref{fig:herr}), where the herringbone structure is now
imposed only in the
central region~$\L_c$.

We proceed by estimating $Z(\partial A)/Z(A)$ as in the proof of
Theorem~\ref{thm:alla1}.
Clearly, we can bound $Z(A)$ from below by
considering only the maximal one-dimensional triangulations $\hat\si
(j)$ in each layer $j$ throughout the whole region~$\L$, and taking
every edge with integer vertical coordinate to be horizontal
(unit length), giving
\[
Z(A)\geq\lambda^{m n^2 + O(mn)}.
\]
Now let $\si\in\partial A$.
The length of an edge whose midpoint is in $\L_\ell\cup\L_r$ is at
most $\e n + m$.
Since there are at most $8\e n m$ such midpoints, the total
contribution of these edges
to $|\si|$ is at most $8(\e n)^2 m + 8\e n m^2$.
Moreover, there must be a midpoint $x_0\in\midpoints_1$ whose
edge~$\sigma_{x_0}$
is vertical (unit length).
The contribution of the layer containing $x_0$ to the total length of
$\si$ is at most
$(un)^2 + ((1-u)n)^2$, where $\e\leq u\leq1-\e$. All other layers
contribute total length
at most $(m-1)n^2$. Hence for any $\si\in\partial A$, we have
\[
|\si|\leq8(\e n)^2 m + 8\e n m^2 + (m-1)n^2
+ (u_*n)^2 + \bigl((1-u_*)n \bigr)^2, 
\]
where $u_*\in[\e,(1-\e)]$ maximizes $(un)^2 + ((1-u)n)^2$. Clearly,
we have
$u_*\sim\e$ [or by symmetry $u_*\sim(1-\e)$].
Since there are at most $2^{3mn}$ triangulations in total~\cite
{Anclin}, we get
%
\[
\frac{Z(\partial A)}{Z(A)}\leq\lambda^{8(\e n)^2 m + 8\e n m^2 -n^2
+ (\e
n)^2 + ((1-\e)n)^2 + O(mn)}.
\]
%
Thus taking $\e= c /m$ and $n\geq m^2/c$ for some sufficiently small
constant $c>0$
(independent of $n,m$)
gives
\[
\frac{Z(\partial A)}{Z(A)}\leq e^{-c' n^2/m}.
\]
This completes the proof.
\end{pf}

\textit{Note}: While the above bound is tight in the case $m=O(1)$, it
becomes progressively
weaker as $m$ increases with~$n$, and vacuous when $m\sim\sqrt{n}$.
In particular, it is far from the conjectured
behavior $\tmix=e^{\O(mn(m+n))}$ stated in Conjecture~\ref{conj:mt}.
The deterioration
of the bound as $m$ increases is apparently an artifact of our proof,
which is based on
an essentially one-dimensional ``herringbone'' structure similar to that
used in the proof of
Theorem~\ref{thm:alla1}.

We end this section with a much weaker lower bound that holds for all
$\lambda>0$. This shows
in particular that the upper bound of Theorem~\ref{gapknbound} for the
small~$\lambda$ regime is tight.

\begin{proposition}\label{prop:alla}
There exists $c>0$ such that the mixing time of the Glauber dynamics on
triangulations
of $\Lambda^0_{m,n}$ in the absence of constraints satisfies $\tmix
\geq c mn(m+n)$
for all $m,n\geq1$ and all $\lambda>0$.
\end{proposition}

\begin{pf}
Assume that $n\geq m$. Let $x = (1/2, n/2)\in\Lambda$, and let the
initial condition $\si$
be an arbitrary triangulation in which $\si_x$ is the edge from $(0,0)$
to $(1,n)$.
Let $A$ be the set of triangulations in which $\sigma_x$ is \textit{not}
positively oriented.
By symmetry, the stationary probability of $A$ is at least $1/2$ (with
strict inequality
if $x$ is of type 1). Then, the total variation distance at time~$t$
can be bounded
below as
\[
\bigl\|p^t(\si,\cdot)-\mu\bigr\|\geq\tfrac{1}2 - p^t(\si,A)
\geq\tfrac{1}2 - \bbP_\si (\t_A\leq t),
\]
where $\t_A$ denotes the hitting time of the set $A$, and $\bbP_\si$ is
the probability
on trajectories of the Markov chain with initial condition~$\si$.
Let $\Lambda' \subseteq\Lambda$ be the set of midpoints $(1/2,i)$ for
$i=0,1/2,1,\ldots,n$.
In order to flip $\si_x$ all the way to an edge that is not positively
oriented, we need
to perform at least $\O(n^2)$ flips of the edges with midpoints
in~$\Lambda'$, as can
be seen easily from the correspondence with lattice paths recalled in the
proof of Theorem~\ref{thm:alla1}.
Thus $\bbP_\si(\t_A\leq t)$ is bounded above by the probability that
out of $t$ updates in $\L$, at least $\O(n^2)$ of them are in $\L'$.
Each update independently falls in $\L'$ with probability $\frac{|\L
'|}{|\L|}=O(1/m)$.
By Markov's inequality it follows that, for some constant $C>0$,
\[
\bbP_\si(\t_A\leq t) \leq C \frac{t}{mn^{2}}.
\]
Taking $t=mn^{2}/4C$ yields $\tmix=\Omega(mn(m+n))$.
\end{pf}

\section{Future work}\label{sec:future}
Our results suggest a number of immediate open questions.
\begin{longlist}[(1)]
\item[(1)] Can the spatial mixing property (Theorem~\ref{thm:zero}) and
the polynomial
mixing time bound (Theorem~\ref{thm:one}) be extended to the entire
regime $\lambda<1$?
This would in particular complete the verification of part~(a) of
Conjecture~\ref{conj:mt}.
These questions, which are related, are likely to require a deeper
investigation of the
geometry of triangulations along the lines begun in Section~\ref
{sec:equilibrium}.
\item[(2)] Can the stronger exponential lower bound on mixing time in
part~(b) of Conjecture~\ref{conj:mt}
be proved in the regime $\lambda>1$? This will require the use of more
sophisticated
rigid triangulations than the ``herringbone'' structures used in
Section~\ref{sec:lowerB}.
\item[(3)] The unweighted case $\lambda=1$ seems particularly
challenging, as our results
strongly suggest that it corresponds to a ``critical point.''
Is the mixing time polynomial in this case?
\item[(4)] In the ``super-critical'' regime $\lambda>1$, the model
appears to exhibit various
``phases'' according to the direction of alignment of the edges. It
would be interesting
to describe these phases and their contributions to the equilibrium
distribution.
\end{longlist}

%



\printaddresses
\end{document}